\documentclass[11pt,a4paper]{amsart}
\usepackage{amssymb}
\usepackage{mathabx}
\usepackage{mathrsfs}
\usepackage{syntonly}
\usepackage{amsmath}
\usepackage{amsthm}
\usepackage{amsfonts}
\usepackage{amssymb}
\usepackage{latexsym}
\usepackage{amscd,amssymb,amsopn,amsmath,amsthm,graphics,amsfonts,mathrsfs,accents,enumerate,verbatim,calc}
\usepackage[dvips]{graphicx}
\usepackage[colorlinks=true,linkcolor=red,citecolor=blue]{hyperref}
\usepackage[all]{xy}

\date{}
\allowdisplaybreaks[4] \footskip=15pt
\renewcommand{\uppercasenonmath}[1]{}

\numberwithin{equation}{section}
 \theoremstyle{plain}
\newtheorem{lem}{Lemma}[section]
\newtheorem{cor}[lem]{Corollary}
\newtheorem{prop}[lem]{Proposition}
\newtheorem{thm}[lem]{Theorem}

\newtheorem{Quest}[lem]{Question}
\newtheorem{Property}[lem]{Property}
\newtheorem{Properties}[lem]{Properties}
\newtheorem{Subprops}{}[lem]
\newtheorem{Para}[lem]{}

\theoremstyle{remark}
\newtheorem{rec}[lem]{Recollection}
\newtheorem{hyp}[lem]{Hypothesis}

\newtheorem{definition}[lem]{Definition}
\newtheorem{Ex}[lem]{Example}
\newtheorem{cons}[lem]{Construction}
\newtheorem{nota}[lem]{Notation}
\newtheorem{rem}[lem]{Remark}

\newenvironment{df}{\begin{definition}\rm}{\end{definition}}
\newenvironment{ex}{\begin{Ex}\rm}{\end{Ex}}
\newenvironment{quest}{\begin{Quest}\rm}{\end{Quest}}

\newtheorem*{ack*}{ACKNOWLEDGEMENTS}

\newcommand{\pf}{\noindent\begin {proof}}
\newcommand{\epf}{\end{proof}}

\begin{document}
\begin{center}
{\large  \bf  THE TELESCOPE CONJECTURE FOR GLOBAL REPRESENTATIONS AND FI-MODULES}

\vspace{0.5cm}  Peng Xu$\footnote{Corresponding author. Peng Xu is supported by the National Natural Science Foundation of China (Grant No. 12231009, 11971224) }$

\end{center}

\bigskip
\centerline { \bf  Abstract}
\medskip

\leftskip10truemm \rightskip10truemm \noindent\hspace{1em}  In this paper, we classify the localizing ideals of the derived category $\mathrm{D}(\mathcal{U})$ of global representations over a  field $k$ of characteristic zero, for various infinite families $\mathcal{U}$ of finite groups. These families include elementary abelian $p$-groups, cyclic $p$-groups, and cyclic groups of prime order together with the trivial group. We deduce that the telescope conjecture holds for these $\mathrm{D}(\mathcal{U})$. In particular, via Pontryagin duality, our results for elementary abelian $p$-groups also establish the telescope conjecture and the corresponding classification for derived $\mathrm{VI}$-modules. We also prove that the telescope conjecture holds for the derived category of $\mathrm{FI}$-modules.  \\[2mm]

\leftskip0truemm \rightskip0truemm
\section { \bf Introduction}
The telescope conjecture is well understood in many rigidly-compactly generated tensor triangulated categories, including the derived category of a commutative noetherian ring \cite{Nee92}, the stable module category of a finite group or finite group scheme \cite{BBIKP25,BIK11,BIKP18}, the derived category of Artin motives over a finite field \cite{BG22, BG25}, and the category of rational $G$-spectra for any compact Lie group $G$ \cite{BHS23}, where Balmer-Favi support and stratification provide powerful mechanisms for controlling smashing tensor ideals. Considerably less is known beyond the rigid setting. In particular, there is no general support theory that classifies localizing tensor ideals in an arbitrary compactly generated tensor triangulated category whose compact objects are not rigid.

In this paper we study this problem for derived categories arising in global representation theory and representation stability. Let $\mathcal{U}$  be a family of finite groups and let $$\mathrm{A}(\mathcal{U})=\mathrm{Fun}(U^{op}, \mathrm{Mod}~k)$$
be  the Grothendieck category of $\mathcal{U}$-global representations over a field $k$ of characteristic zero.  Its derived category $\mathrm{D}(\mathcal{U})$ is compactly generated and symmetric monoidal, but is generally not rigidly-compactly generated. These categories form the algebraic counterpart of rational global homotopy theory and are closely connected with $\mathrm{FI}$ and $\mathrm{VI}$ modules.

Our main results determine the large tensor triangular geometry of these categories. We classify all localizing tensor ideals of $\mathrm{D}(\mathcal{U})$ for three infinite families of finite groups: cyclic groups of prime order together with the trivial group, cyclic $p$-groups, and elementary abelian $p$-groups. We also classify all localizing tensor ideals in the derived category of $\mathrm{FI}$-modules. In each case the classification is governed by homological support. As a consequence, every smashing tensor ideal is generated by compact objects, and hence the telescope conjecture holds.

These results go beyond the classification of thick tensor ideals in the compact subcategory. Indeed, the passage from $\mathrm{D}(\mathcal{U})^{c}$ to the full derived category $\mathrm{D}(\mathcal{U})$ is not formal in the absence of rigidity, and the usual Balmer-Favi and stratification machinery does not apply. Our approach instead uses characteristic objects, the torsion theory of the underlying Grothendieck category, and a tensor-minimal generic quotient. This produces a torsion-generic decomposition of localizing ideals and allows homological support to control arbitrary, possibly non-compact, objects.

A further feature is that the geometry governing the big derived category need not coincide directly with the Balmer geometry of its compact objects. For elementary abelian $p$-groups and $\mathrm{FI}$-modules, the classification is naturally expressed in terms of the tensor abelian spectrum of the noetherian heart. Thus the results also exhibit a connection between tensor abelian geometry and localizing tensor ideals in non-rigid derived categories.

The classification of all localizing ideals is substantially stronger than the telescope conjecture itself. It determines not only the smashing localizations but the entire lattice of tensor-closed localizing subcategories. Moreover, it identifies this lattice with an explicit topological or tensor-abelian parameter space. Theorem \ref{thm3.11} provides a complete classification of localizing ideals of $\mathrm{D}(\mathfrak{C}_{p})$, where $\mathfrak{C}_{p}$ denotes the family of cyclic groups of prime order together with the trivial group.

\vspace{2mm}
 {\bf Theorem 1.1 }  The homological support induces a bijection  \[
\{\text{localizing ideals of } \mathrm{D}(\mathfrak{C}^{\mathrm{pr}})\}
\simeq
\{\text{subsets of } \mathrm{Spc}(\mathrm{D}(\mathfrak{C}^{\mathrm{pr}})^{c}\}.
\]
For the noetherian family $\mathfrak{C}_{p}$ of cyclic $p$-groups, Theorem \ref{thm4.15} shows there is a bijection between the localizing ideals of $\mathrm{D}(\mathfrak{C}_{p})$ and the open subsets of $\mathrm{Spc}(\mathrm{D}(\mathfrak{C}_{p})^{c})$.

\vspace{2mm}
 {\bf Theorem 1.2 }The homological support induces the following bijections:
\[
\{\text{localizing ideals of } \mathrm{D}(\mathfrak{C}_{p})\}
\simeq
\{\text{localizing ideals of } \mathrm{A}(\mathfrak{C}_{p})\}
\simeq
\{\text{open subsets of } \mathrm{Spc}(\mathrm{D}(\mathfrak{C}_{p})^{c}\}.
\]
The localizing ideals of $\mathrm{D}(E_{p})$ and of the derived category of $\mathrm{FI}$-modules are classified by Theorems \ref{thm5.3} and \ref{thm6.13}. Here $E_{p}$ denotes the family of elementary abelian $p$-groups. Note that the Balmer spectrum of $\mathrm{D}(E_{p})^{c}$ is not homeomorphic to $\mathbb{N}^{*}$; instead, we use the tensor abelian spectrum (see \cite{Xu26}).

\vspace{2mm}
 {\bf Theorem 1.3 }Let $\mathcal{A}$ be $\mathrm{A}(E_{p})$ or the category of $\mathrm{FI}$-modules, with $\mathcal{A}^{c}$ the subcategory of finitely generated objects and $\mathcal{D}$ the derived category of $\mathcal{A}$. Then the homological support induces the following bijections:
\[
\{\text{localizing ideals of } \mathcal{D}\}
\simeq
\{\text{localizing ideals of } \mathcal{A}\}
\simeq
\{\text{open subsets of } \mathrm{Spc}(\mathcal{A})^{c}\}.
\]
Although the resulting classifications have a common support-theoretic form, the mechanisms are not identical. For $\mathfrak{C}^{\mathrm{pr}}$, arbitrary subsets occur. For cyclic $p$-groups, the generic point imposes a cofiniteness condition. For elementary abelian $p$-groups and $\mathrm{FI}$-modules, the appropriate parameter space is the tensor abelian spectrum rather than the Balmer spectrum of compact objects. These differences illustrate how the large tensor-triangular geometry depends both on the asymptotic representation theory of the indexing category and on the relation between the compact and noetherian parts. 

In each of these cases, our classification also identifies the smashing ideals and shows that they are generated by compact objects. We therefore obtain the following consequence.

\vspace{2mm}
 {\bf Theorem 1.4 }  The telescope conjecture holds for $\mathrm{D}(\mathfrak{C}^{\mathrm{pr}})$, $\mathrm{D}(\mathfrak{C}_{p})$, $\mathrm{D}(E_{p})$ and the derived category of $\mathrm{FI}$-modules.

Note that Pontryagin duality induces an equivalence between the category of $\mathrm{VI}$-modules (over $k$) and $\mathrm{A}(E_{p})$. Consequently, our results also classify the localizing tensor ideals of the derived category of $\mathrm{VI}$-modules and establish the telescope conjecture for the derived categories of both $\mathrm{VI}$-modules and $\mathrm{FI}$-modules. This leads naturally to the following question: does $\mathrm{D}(\mathcal{U})$ satisfy the telescope conjecture for every global family $\mathcal{U}$?

\textbf{Conventions}: Throughout, we  will work over a field $k$ of characteristic 0. We write $G\twoheadrightarrow H$ to denote a surjective group homomorphism. All categories considered in this paper are assumed to be small. By a subcategory, we always mean one that is replete and full.

\section { \bf Preliminaries}
This section contains some brief reminders of various results, constructions, and topics that will be used in the sequel together with some pertinent references. We begin with recollections on the global representation theory and FI-modules. We then review some  results on subcategories of compactly generated tensor triangulated categories.

\subsection{Global representation theory}
We refer the reader to \cite{BBP+25a, BBP+25b,PS22} for more details on global representation theory.

Throughout this paper,  let $\mathcal{G}$ be the category of finite groups and conjugacy classes of surjective group homomorphisms. A replete full subcategory $\mathcal{U}$ of $\mathcal{G}$ is called closed downwards if for any surjective group homomorphism $G\twoheadrightarrow H$, $G\in \mathcal{U}$ implies $H\in \mathcal{U}$. $\mathcal{U}$  is called closed upwards if for any surjective group homomorphism $H\twoheadrightarrow G$, $G\in \mathcal{U}$ implies $H\in \mathcal{U}$. $\mathcal{U}$  is called widely closed  if whenever $G\twoheadleftarrow H\twoheadrightarrow K$ are surjective homomorphisms with $G,H,K\in \mathcal{U}$, the image of the combined morphism $H\rightarrow G\times K$ is also in $\mathcal{U}$. In particular, if $\mathcal{U}$ is closed downwards or closed upwards, then $\mathcal{U}$ is widely closed. $\mathcal{U}$ is called essentially finite if it contains only finitely many isomorphism classes of objects. We write $\pi_{0}(\mathcal{U})$ for the set of isomorphism classes of objects of $\mathcal{U}$. We denote by $\mathcal{U}_{> n}$ the full subcategory of $\mathcal{U}$ consisting of groups of order greater than
n, and by $\mathcal{U}_{\leq n}$ the full subcategory consisting of groups of order at most n.  The upwards closure of a subset $S\subset \pi_{0}(\mathcal{U})$ is $$\uparrow(S):=\{~[G]\in \pi_{0}(\mathcal{U}) |~\exists [H]\in S:~G\twoheadrightarrow H\}.$$

\begin{rec}
 Given a replete full subcategory $\mathcal{U}\subseteq \mathcal{G}$, we denote by $$\mathrm{A(\mathcal{U}):=Fun(\mathcal{U}^{op};Mod}~k)$$ the abelian category of functors $\mathrm{\mathcal{U}^{op}\rightarrow Mod}~k.$ Let $X\in \mathrm{A(\mathcal{U})}$, the (index) support of $X$ is defined to be $\mathrm{supp}(X):=\{G\in\pi_{0}(\mathcal{U})|X(G)\neq 0\}$. When $\{G\}$ is the replete full subcategory spanned by a single finite group $G$, we have $\mathrm{A}(\{G\})\simeq \mathrm{Mod}~k[\mathrm{Out}(G)]$.

Let $\mathcal{U}$ be a subcategory and $G\in \mathcal{U}$. There is an evaluation functor $$ev_{G}:\mathrm{A}(\mathcal{U})\rightarrow \mathrm{A}(\{G\})\quad\quad\quad X\mapsto X(G)$$ which admits a left adjoint $$e_{G,\bullet}:\mathrm{A}(\{G\})\rightarrow \mathrm{A}(\mathcal{U})\quad\quad\quad V\mapsto e_{G,V}\cong V\otimes_{k[\mathrm{Out}(G)]}k[\mathrm{Hom}_{\mathcal{U}}(-,G)].$$ Set $e_{G}:=e_{G,k[\mathrm{Out}(G)]}$, then we have $\mathrm{Hom}_{\mathrm{A}(\mathcal{U})}(e_{G},X)\cong X(G)$. Let V be an $\mathrm{Out}(G)$-representation, the object $\chi_{G,V}$ is defined as follows: $\chi_{G,V}(H)=e_{G,V}(H)$ if $G\cong H$ and  $\chi_{G,V}(H)=0$ otherwise. There is a map $e_{G,V}\rightarrow \chi_{G,V}$ induced by the identity map of $V$.

We have the following facts about $\mathrm{A}(\mathcal{U})$.
\begin{enumerate}
\item $\mathrm{A}(\mathcal{U})$ is a Grothendieck abelian category with generators given by $e_{G}$ for each $G\in \mathcal{U}$.

\item $\mathrm{A}(\mathcal{U})$ is a symmetric monoidal category with pointwise tensor product, i.e.,$(X\otimes Y)(G)=X(G)\otimes Y(G)$ for $X,Y\in \mathrm{A}(\mathcal{U})$ and $G\in \mathcal{U}$. The tensor unit $\mathrm{1}$ is the constant functor with value $k$ and all maps the identity. When $\mathcal{U}$ contains the trivial group, then $\mathrm{1}\cong e_{1}$.

\item Each object in $\mathrm{A}(\mathcal{U})$ is flat.

\end{enumerate}
Recall that an object $X\in \mathrm{A}(\mathcal{U})$ is finitely generated if there exists an epimorphism $\oplus^{n}_{i=1}e_{G_{i}}\twoheadrightarrow X$. For $X\in \mathrm{A}(\mathcal{U})$ and $x\in X(G)$,  $x$ is called torsion if there exists $\alpha:H\twoheadrightarrow G$ such that $\alpha^{*}(x)=0$, $x$ is called torsion-free if it is not torsion. $X$ is called torsion if it consists of torsion elements.
\end{rec}

\begin{df}Let $\mathcal{U}\subseteq \mathcal{G}$ be a subcategory. The derived category of global representations is $$\mathrm{D}(\mathcal{U}):=\mathrm{D}(\mathrm{A}(\mathcal{U})).$$
\end{df}

\begin{cons}Note that the functor $e_{G,\bullet}$ is exact. It then follows that the adjoint pair $(e_{G,\bullet}, ev_{G})$ descends to an adjunction at the level of derived categories without deriving the functors. For $V\in \mathrm{D}(\{G\})$, we let $\chi_{G,V}$ denote the object $$\chi_{G,V}:=e_{G,V}\otimes \chi_{G,k}.$$
\end{cons}
\begin{hyp}\label{hyp2.3}Let $\mathcal{U}\subseteq \mathcal{G}$ be a subcategory which is full, replete, and widely closed, and for which the tensor unit of $\mathrm{D}(\mathcal{U})$ is compact.
\end{hyp}

\begin{prop}Let $\mathcal{U}\subseteq \mathcal{G}$ be a subcategory satisfying Hypothesis \ref{hyp2.3}. Then the objects $\{e_{G}~|~G\in \mathcal{U}\}$ form an essentially small class of compact objects that generates the derived category $\mathrm{D}(\mathcal{U})$. Moreover, the subcategory of compact objects $\mathrm{D}(\mathcal{U})^{c}$ is  symmetric monoidal.
\end{prop}
\begin{proof} See \cite[Proposition 2.2 and 2.4]{BBP+25b}.
\end{proof}
\begin{df} Let $X\in \mathrm{D}(\mathcal{U})$. The homological support of $X$ is $$\mathrm{hsupp}(X):=\{G\in \pi_{0}(\mathcal{U})~|~X(G)\not\simeq 0\}.$$ Note that $X(G)\not\simeq 0$ if and only if $H_{*}(X)(G)\neq 0$.
\end{df}
\begin{cons}
Given a functor $f:\mathcal{U}\rightarrow \mathcal{V}$, we have adjunctions
\[
\xymatrix@R=0em@C=4em{
\mathsf{A}(\mathcal{U})
\ar@/^3ex/[r]^{f_{!}}        
\ar@/_3ex/[r]_{f_{*}}        
&
\mathsf{A}(\mathcal{V})
\ar[l]^{f^{*}}               
}
\]
where $f^{*}$ is defined by $f^{*}(X)(H)=X(f(H))$, the left adjoint $f_{!}$ of $f^{*}$ is given by the left Kan extension along $f$, and the right adjoint $f_{*}$ of $f^{*}$ is given by the right Kan extension along $f$. Passing to the derived categories, we have the following adjunctions \[
\xymatrix@R=0em@C=4em{
\mathsf{D}(\mathcal{U})
\ar@/^3ex/[r]^{f_{!}}        
\ar@/_3ex/[r]_{f_{*}}        
&
\mathsf{D}(\mathcal{V}).
\ar[l]^{f^{*}}               
}
\]
\end{cons}
\subsection{The category of $\mathrm{FI}$-modules} For more details on the theory of $\mathrm{FI}$-modules, see \cite{CEF15,CEFN14}.
\begin{rec}Let $\mathrm{FI}$ be the category whose objects are the finite sets $$[n]=\{1,2,\cdots,n\}~~~~~\quad\quad (n\geq 0)$$ and whose morphisms are injections.  An $\mathrm{FI}$-module over the field $k$ is a covariant  functor $$X:\mathrm{FI}\rightarrow \mathrm{Mod}~k.$$ We denote by $\mathrm{Mod}_{\mathrm{FI}}$ the  abelian category of $\mathrm{FI}$-modules.
To simplify notation, we write $X(n)$ instead of $X([n])$. An element  $x\in X(n)$ is called torsion if there exists an injection $\alpha:[n]\rightarrow [m]$ such that $\alpha_{*}(x)=0$. We say $x$ is torsion-free if it is not torsion and say $X$ is torsion if it consists of torsion elements. For an $S_{n}$-representation $V$, define $$M_{n,V}:=k[\mathrm{Inj}([n],-)]\otimes_{k[S_{n}]}V.$$ Thus $M_{n,V}(m)=0$ if $m<n$ and $M_{n,V}(n)\cong V$. $M_{n,V}$ is a finitely generated projective object in $\mathrm{Mod}_{\mathrm{FI}}$ when $V$ is finite-dimensional. Given any $\mathrm{FI}$-module $X$, we  have $$\mathrm{Hom}_{\mathrm{FI}}(M_{n,V},X)\cong \mathrm{Hom}_{k[Sn]}(V,X(n)).$$ When $V=k[S_{n}]$, we write $M_{n}$ for $M_{n,k[S_{n}]}$. It is known that $M_{n}$ is a direct summand of $M_{n}\otimes M_{m}$ for $n\geq m$. Note that each finitely generated projective $\mathrm{FI}$-module is a direct summand of a finite direct sum $\oplus_{i=1}^{s}M_{n_{i}}$.

We have the following facts about $\mathrm{Mod}_{\mathrm{FI}}$.
\begin{enumerate}
\item $\mathrm{Mod}_{\mathrm{FI}}$ is a locally Noetherian Grothendieck abelian category with generators given by $M_{n}$ for each $n\in \mathbb{N}$.

\item $\mathrm{Mod}_{\mathrm{FI}}$ is a symmetric monoidal category with pointwise tensor product,i.e.,$(X\otimes Y)(n)=X(n)\otimes Y(n)$ for $X,Y\in \mathrm{Mod}_{\mathrm{FI}}$ and $n\in \mathbb{N}$. $M_{0}$ is the tensor unit $\mathrm{1}$.

\item Each object in $\mathrm{Mod}_{\mathrm{FI}}$ is flat.

\end{enumerate}
\end{rec}
We denote by  $\mathrm{D}_{\mathrm{FI}}$  the derived category of $\mathrm{Mod}_{\mathrm{FI}}$. The following result shows  $\mathrm{D}_{\mathrm{FI}}$ is a compactly generated tensor triangulated category.
\begin{prop} \label{prop2.8} The objects $\{M_{n}~|~n\in \mathbb{N}\}$ form an essentially small class of compact objects that generates $\mathrm{D}_{\mathrm{FI}}$.
\end{prop}
\begin{proof}Since $M_{n}$ is projective for each $n\in \mathbb{N}$, we have a natural isomorphism $\mathrm{Hom}_{\mathrm{D}_{\mathrm{FI}}}(\Sigma^{i}M_{n},X)\cong H_{i}(X(n))$ for each $i\in \mathbb{Z}$ and $X\in \mathrm{D}_{\mathrm{FI}}$. Thus $M_{n}$ is compact for each $n\in \mathbb{N}$. If  $\mathrm{Hom}_{\mathrm{D}_{\mathrm{FI}}}(\Sigma^{i}M_{n},X)\cong H_{i}(X(n))=0$ for each $i\in \mathbb{Z}, n\in \mathbb{N}$, then we have $X\cong 0$, which implies that $\{M_{n}~|~n\in \mathbb{N}\}$ generates $\mathrm{D}_{\mathrm{FI}}$.
\end{proof}
\begin{df} An object $X\in \mathrm{D}_{\mathrm{FI}}$ is called perfect if it is a bounded complex of finitely generated projective $\mathrm{FI}$-modules. We write $\mathrm{D}_{\mathrm{FI}}^{c}$ for the subcategory of perfect complexes.
\end{df}
\begin{prop} For $X\in \mathrm{D}_{\mathrm{FI}}$, the following are equivalent:
\begin{enumerate}
\item $X$ is isomorphic to a  perfect complex in $\mathrm{D}_{\mathrm{FI}}$;

\item $X$ is a compact object in $\mathrm{D}_{\mathrm{FI}}$;

\item $X$ lies in the thick subcategory of $\mathrm{D}_{\mathrm{FI}}$ generated by $\{M_{n}~|~n\in \mathbb{N}\}$.

\end{enumerate}
\end{prop}
\begin{proof} Note that each finitely generated projective $\mathrm{FI}$-module is a direct summand of a finite direct sum $\oplus_{i=1}^{s}M_{n_{i}}$, thus $(1)$ and $(3)$ are equivalent. $(2)$ and $(3)$ are equivalent by \cite[Lemma 2.2]{BKS} and Proposition \ref{prop2.8}.
\end{proof}

\subsection{The telescope conjecture for compactly generated tensor triangulated categories} We assume some familiarity with the basic notions of tensor triangular geometry, as developed by Balmer in \cite{Bal05}.
\begin{rec} Let $\mathcal{J}$ be a compactly generated triangulated category with compact part $\mathcal{J}^{c}$. A thick subcategory $\mathcal{S}$ is called localizing if it is closed under coproducts.  A localizing subcategory $\mathcal{S}$ is called a Bousfield subcategory if the Verdier quotient $q:\mathcal{J}\rightarrow \mathcal{J}/\mathcal{S}$ has a right adjoint. Moreover, if $\mathcal{S}^{\bot}$ is localizing, $\mathcal{S}$ is called a smashing subcategory. Here $\mathcal{S}^{\bot}:=\{Y\in \mathcal{J}~|~\mathrm{Hom}_{\mathcal{J}}(X,Y)=0, \forall X\in \mathcal{S}\}$.\\

A compactly generated tensor triangulated category $(\mathcal{J},\otimes,1)$ is a compactly generated triangulated category $\mathcal{J}$ with a tensor product $\otimes:\mathcal{J}\times \mathcal{J}\rightarrow \mathcal{J}$ which is a triangulated functor in each variable and also makes $\mathcal{J}$ symmetric monoidal with unit $1$. We require $\otimes$ to preserve coproducts. A subcategory $\mathcal{S}$ is called a tensor ideal if $X\otimes Y\in \mathcal{S}$ for $X\in \mathcal{S}$, $Y\in \mathcal{J}$. A localizing subcategory is called a localizing  ideal if it is a tensor ideal. Similarly, a smashing subcategory is called a smashing  ideal if it is a tensor ideal. Let $\mathcal{C}$ be a collection of objects in $\mathcal{J}$. We denote by $\mathrm{Loc}\langle\mathcal{C}\rangle$ the localizing subcategory generated by $\mathcal{C}$, and by $\mathrm{Loc}_{\otimes}\langle\mathcal{C}\rangle$ the localizing ideal generated by $\mathcal{C}$. If $\mathcal{C}$ is a collection of compact objects, we denote by $\mathrm{thick}\langle\mathcal{C}\rangle$ the thick subcategory generated by $\mathcal{C}$, and by $\mathrm{thick}_{\otimes}\langle\mathcal{C}\rangle$ the thick ideal generated by $\mathcal{C}$.
\end{rec}

\begin{prop}\label{Prop2.12}Let $\mathcal{J}$ be a compactly generated tensor triangulated category such that $\mathcal{J}^{c}$ is a symmetric monoidal subcategory of $\mathcal{J}$. If $\mathcal{C}$ is a collection of compact objects, then we have the following:
\begin{enumerate}
\item $\mathrm{Loc}\langle\mathcal{C}\rangle$ is a smashing subcategory of $\mathcal{J}$;

\item $\mathrm{Loc}_{\otimes}\langle\mathcal{C}\rangle$ is a smashing ideal of $\mathcal{J}$.

\end{enumerate}

\end{prop}
\begin{proof}Part (1) is standard, see, for example, \cite[Theorem 1.7]{HR17}. Then (2) comes from the equality $\mathrm{Loc}_{\otimes}\langle\mathcal{C}\rangle=\mathrm{Loc}\langle\mathcal{C},\mathcal{C}\otimes \mathcal{J}^{c}\rangle$ and part (1).
\end{proof}
\begin{rem} Under the assumptions of Proposition \ref{Prop2.12}, a smashing ideal is compactly generated as a localizing ideal if and only if it is compactly generated as a localizing subcategory.
\end{rem}

\begin{df} Let $\mathcal{J}$ be a compactly generated tensor triangulated category. We say that $\mathcal{J}$ satisfies the telescope conjecture if every smashing ideal of $\mathcal{J}$ is compactly generated.
\end{df}

\begin{rem}Our definition of the telescope conjecture coincides with \cite[Definition 9.1]{BHS23} when $\mathcal{J}$ is  a rigidly-compactly generated tensor triangulated category in the sense of \cite{BF11}. In a rigidly-compactly generated tensor triangulated category $\mathcal{J}$, there is a bijection between smashing ideals and tensor-idempotent triangles, which is crucial for developing the stratification theory (see, for example, \cite{BHS23,BIK11}). The main consequence of a stratified category $\mathcal{J}$ (under some mild conditions) is that it satisfies the telescope conjecture. However,  for a general compactly generated tensor triangulated category, there may exist more smashing ideals than tensor-idempotent triangles, see \cite[Example 2.8]{S26}.
\end{rem}

 \section{\bf Global representations of essentially finite families and $\mathfrak{C}^{\mathrm{pr}}$}
The first goal of this section is to classify  the localizing ideals of $\mathrm{D}(\mathcal{U})$ for essentially finite collections $\mathcal{U}$. We show that when $\mathcal{U}$ is essentially finite, there is a bijection between the localizing ideals of $\mathrm{D}(\mathcal{U})$ and the subsets of $\pi_{0}(\mathcal{U})$. Then we classify the localizing ideals of $\mathrm{D}(\mathfrak{C}^{\mathrm{pr}})$, where $\mathfrak{C}^{\mathrm{pr}}$ denotes the family of cyclic groups of prime order together with the trivial group. We prove that the telescope conjecture holds in both cases.\\

Recall that $k$ is always assumed to be a field of characteristic 0.
\begin{lem}\label{lem3.1} Let $\mathcal{U}\subseteq \mathcal{G}$ be a subcategory and let $G\in \mathcal{U}$. Then for any nonzero $W\in \mathrm{D}(\{G\})$, we have $$\mathrm{Loc}_{\otimes}\langle \chi_{G,W}\rangle=\mathrm{Loc}_{\otimes}\langle \chi_{G,k}\rangle \subseteq \mathrm{D}(\mathcal{U}).$$
\end{lem}
\begin{proof} Since $k[\mathrm{Out}(G)]$ is semisimple and $\mathrm{D}(\{G\})\cong \mathrm{D}(k[\mathrm{Out}(G)])$, we have $W\cong \bigoplus_{n\in \mathbb{Z}}H_{n}(W)[-n]$. Fix some $n$ such that $V:=H_{n}(W)\neq0$. Then $k$ is a direct summand of $V\otimes V^{*}$. This implies that $k\in \mathrm{Loc}_{\otimes}\langle W\rangle$. For the inclusion $\mathrm{Loc}_{\otimes}\langle \chi_{G,W}\rangle \supseteq \mathrm{Loc}_{\otimes}\langle \chi_{G,k}\rangle $, note that the tensor exact functor $\chi_{G,-}:\mathrm{D}(\{G\})\rightarrow \mathrm{D}(U)$ preserves coproducts. Therefore $$\chi_{G,k}\in \chi_{G,-}(\mathrm{Loc}_{\otimes}\langle W\rangle)\subseteq \mathrm{Loc}_{\otimes}\langle \chi_{G,W}\rangle.$$
The other containment follows from the identity $\chi_{G,W}\cong \chi_{G,k}\otimes\chi_{G,W}$.
\end{proof}

The following result is a localizing analogue of \cite[Proposition 5.2]{BBP+25b}. We provide the proof for convenience.
\begin{prop}\label{prop3.2} Let $\mathcal{U}$ be essentially finite and $X\in \mathrm{D}(\mathcal{U})$. Then $$\mathrm{Loc}_{\otimes}\langle X\rangle=\mathrm{Loc}_{\otimes}\langle \chi_{G,k}~|~G\in \mathrm{hsupp}(X)\rangle.$$ In particular, for $Y\in \mathrm{D}(\mathcal{U})$, we have $\mathrm{hsupp}(X)\subseteq \mathrm{hsupp}(Y)$  if and only if $X\in \mathrm{Loc}_{\otimes}\langle Y\rangle$.
\end{prop}
\begin{proof} If $G\in \mathrm{hsupp}(X)$, there is a quasi-isomorphism $$\chi_{G,X(G)}:=\chi_{G,k}\otimes e_{G,X(G)}\xrightarrow{\simeq} \chi_{G,k}\otimes X .$$ Therefore, by Lemma \ref{lem3.1}, we have $$\chi_{G,k}\in \mathrm{Loc}_{\otimes}\langle \chi_{G,X(G)}\rangle\in \mathrm{Loc}_{\otimes}\langle X\rangle,$$ so the $"\supseteq"$ containment holds.

For the other containment, using that $\pi_{0}(\mathcal{U})$ is finite, we prove by induction on $n=|~\mathrm{hsupp}(X)|$ that $X\in \mathrm{Loc}_{\otimes}\langle \chi_{G,k}~|~G\in \mathrm{hsupp}(X)\rangle$. If $n=0$ then $X\simeq 0$ and the claim is clear. Now let $n\geq 1$ and suppose that the claim holds for all objects $Z\in \mathrm{D}(\mathcal{U})$ with $|~\mathrm{hsupp}(Z)|< n$. Let $\{G_{1},\cdots, G_{n}\}$ be a complete collection of representatives for the isomorphism classes of groups in $\mathrm{hsupp}(X)$ and choose $G_{n}$ maximal among these groups with respect to $\twoheadrightarrow$. Set $V_{n}=X(G_{n})$. Consider the following diagram $$\xymatrix@C=0.5cm{
  F \ar[r]\ar[d] & e_{G_{n},X(G_{n})} \ar[r]\ar[d]^{g_{n}} & \chi_{G_{n},X(G_{n})} \ar@{-->}[d]^{f_{n}} \\
  0 \ar[r] & X \ar[r]^{\cong} & X  }$$
  where $F\in \mathrm{Loc}\langle e_{H}~| H\in \uparrow (\{G_{n}\})\setminus\{G_{n}\} \rangle$ by \cite[Lemma 2.16]{BBP+25b}, the upper row is a triangle and $g_{n}$ is the map induced by the identity map of $X(G_{n})$. By construction we have $\mathrm{hsupp}(X)\cap \mathrm{hsupp}(F)=\emptyset$, thus the left square is commutative and so there exists $f_{n}: \chi_{G_{n},X(G_{n})}\rightarrow X$ to make the right square commutative. Note that the map $f_{n}$ is a quasi-isomorphism at $G_{n}$ and so $\mathrm{hsupp}(C)=\mathrm{hsupp}(X)\setminus\{G_{n}\}$. Now let $C=\mathrm{cone}(f_{n})$, we have $X\in \mathrm{Loc}_{\otimes}\langle \chi_{G_{n},V_{n}}, C\rangle$. By the induction hypothesis $$X\in \mathrm{Loc}_{\otimes}\langle \chi_{G_{n},V_{n}}, C\rangle\subseteq \mathrm{Loc}_{\otimes}\langle \chi_{G_{n},V_{n}}, \chi_{G_{i},k}~|~1\leq i<n\rangle.$$ To conclude the proof, we apply Lemma \ref{lem3.1} again.
The final statement is an immediate consequence of the first claim and \cite[Lemma 4.8]{BBP+25b}.
\end{proof}

\begin{df} Let $\mathcal{U}\subseteq \mathcal{G}$ be a subcategory. If $\mathcal{L}\subseteq \mathrm{D}(\mathcal{U})$ is a localizing ideal, we define $$S(\mathcal{L}):=\{[G]\in \pi_{0}(\mathcal{U})\mid\chi_{G,k}\in \mathcal{L}\}.$$ We say $\mathcal{L}$ is torsion if $\mathcal{L}=\mathrm{Loc}_{\otimes}\langle \chi_{G,k}~|~G\in S(\mathcal{L})\rangle$.
\end{df}
\begin{thm}\label{thm3.4}Let $\mathcal{U}$ be essentially finite and satisfy Hypothesis \ref{hyp2.3}. Then every localizing ideal is torsion and the homological support induces  bijections \[
\{\text{localizing ideals of } \mathrm{D}(\mathcal{U})\}
\simeq
\{\text{subsets of } \pi_{0}(\mathcal{U})\}
\simeq
\{\text{subsets of } \mathrm{Spc}(\mathrm{D}(\mathcal{U})^{c})\}.
\]
\end{thm}
\begin{proof} The first claim comes from Proposition \ref{prop3.2} directly. For any localizing ideals $\mathcal{L}_{1}$, $\mathcal{L}_{2}$, we have $$\mathcal{L}_{1}\subseteq \mathcal{L}_{2}\Longleftrightarrow \mathrm{hsupp}(\mathcal{L}_{1})\subseteq \mathrm{hsupp}(\mathcal{L}_{2}).$$ In particular, every localizing ideal is determined by its homological support and the converse holds. This implies the first bijection. The second bijection follows from \cite[Theorem 5.3]{BBP+25b}.
\end{proof}

\begin{cor}\label{cor3.5}Let $\mathcal{U}$ be essentially finite and satisfy Hypothesis \ref{hyp2.3}. Then the telescope conjecture holds for $\mathrm{D}(\mathcal{U})$.
\end{cor}
\begin{proof} By  \cite[Lemma 5.1]{BBP+25b}, the object $\chi_{G,k}$ is compact in $\mathrm{D}(\mathcal{U})$ for each $G\in \mathcal{U}$. Then the result comes from Theorem \ref{thm3.4}.
\end{proof}

In the rest of this section, we will study the localizing ideals of $\mathrm{D}(\mathfrak{C}^{\mathrm{pr}})$, where $\mathfrak{C}^{\mathrm{pr}}$ is the family of cyclic groups of prime order together with the trivial group. The Balmer spectrum of $\mathrm{D}(\mathfrak{C}^{\mathrm{pr}})^{c}$
has been calculated by Barrero, Barthel, Pol, Strickland and Williamson in \cite{BBP+25b} via the reflective filtration of $\mathfrak{C}^{\mathrm{pr}}$. Moreover, their method could be applied to a large class of families, for example, the family of cyclic $p$-groups. We will provide a new elementary method for classifying  prime ideals in $\mathrm{D}(\mathfrak{C}^{\mathrm{pr}})^{c}$. This will help to understand the smashing ideals of $\mathrm{D}(\mathfrak{C}^{\mathrm{pr}})$.

\begin{nota}We denote by $\mathbb{P}$  the set of primes. For simplicity, we write $e_{p,V}$ for $e_{C_{p},V}$,~$X(p)$ for $X(C_{p})$ and so on. We write  $\Gamma_{p}:=\mathrm{Out}(C_{p})=(\mathbb{Z}/p)^{\times}$.
\end{nota}

\begin{lem}\label{lem3.7} Let $X\in \mathrm{D}(\mathfrak{C}^{\mathrm{pr}})^{c}$. Then there exists a finite subset $F\subseteq \mathbb{P}$ such that for $q\notin F$, the natural morphism $X(1)\rightarrow X(q)$ is an isomorphism in $\mathrm{D}(k[\Gamma_{q}])$, where $X(1)$ is equipped with the trivial $\Gamma_{q}$-action.
\end{lem}

\begin{proof}Since $\chi_{p,k}=e_{p,k}$  for $p\in \mathbb{P}$, it follows from \cite[Theorem 7.3]{BBP+25a} that $$\mathrm{D}(\mathfrak{C}^{\mathrm{pr}})^{c} = \bigcup_{\substack{F\subseteq \mathbb{P} \\ F \text{ finite}}} \operatorname{thick}\langle e_{1}, e_{p} \mid p\in F\rangle.$$ Thus we can choose some finite $F\subseteq \mathbb{P}$ such that $X\in {\mathrm{thick}}\langle e_{1}, e_{p} \mid p\in F\rangle$. For any $q\notin F$, define the subcategory $$\mathcal{E}_{q}:=\{Y\in \mathrm{D}(\mathfrak{C}^{\mathrm{pr}})^{c}\mid Y(1)\xrightarrow{\simeq} Y(q)\}.$$ It is easy to show that $\mathcal{E}_{q}$ is a thick subcategory containing $e_{1}$ and $e_{p}$ for $p\in F$. Thus we have $X\in {\mathrm{thick}}\langle e_{1}, e_{p} \mid p\in F\rangle\subseteq \mathcal{E}_{q}$, which concludes the proof.
\end{proof}

\begin{df}\label{df3.8} For every finite subset $F\subseteq \mathbb{P}$, we define $\gamma_{F}$ to be the cokernel of the monomorphism $\bigoplus_{p\in F}e_{p,k}\rightarrowtail e_{1}$ in $\mathrm{A}(\mathfrak{C}^{\mathrm{pr}})$. When $F=\{p\}$, we write $\gamma_{p}$ for simplicity.
\end{df}

Recall that the support datum $(\pi_{0}(\mathcal{U}), \mathrm{hsupp})$ induces group primes $P_{q}=\{X\in \mathrm{D}(\mathfrak{C}^{\mathrm{pr}})^{c}\mid X(q)\simeq 0\}$ for $q\in \{1\}\bigcup \mathbb{P}$.
\begin{lem}\label{lem3.9}The group primes in $\mathrm{D}(\mathfrak{C}^{\mathrm{pr}})^{c}$ have the following forms:
\begin{enumerate}
\item $P_{p}=\mathrm{thick}_{\otimes}\langle\gamma_{p}\rangle$ for $p\in \mathbb{P}$;

\item $P_{1}=\mathrm{thick}_{\otimes}\langle e_{p}\mid p\in \mathbb{P}\rangle$ .

\end{enumerate}
\end{lem}
\begin{proof}Since $\gamma_{p}(p)=0$, $\gamma_{p}\in P_{p}$. Conversely, let $X\in P_{p}$. Then $X(p)=0$ by definition, and the identity $\gamma_{p}\otimes X\cong X$ implies $X\in \mathrm{thick}_{\otimes}\langle\gamma_{p}\rangle$. This proves (1).
For part (2), the $"\supseteq"$ containment is obvious. For the other containment, suppose $X\in P_{1}$.  Then by Lemma \ref{lem3.7}, there exists some finite subset $F\subseteq \mathbb{P}$ such that  $X(q)=0$ for $q\notin F$. Applying $X\otimes -$ to the triangle $$\bigoplus_{p\in F}e_{p,k}\rightarrow 1\rightarrow \gamma_{F},$$  we have $X\cong \bigoplus_{p\in F}e_{p,k}\otimes X$ since $\gamma_{F}\otimes X\simeq 0$. This implies that $X\in \mathrm{thick}_{\otimes}\langle e_{p}\mid p\in \mathbb{P}\rangle$, which concludes the proof.
\end{proof}

\begin{prop}\label{prop3.10} Every prime ideal of $\mathrm{D}(\mathfrak{C}^{\mathrm{pr}})^{c}$ is of the form $P_{1}$ or $P_{p}$ for some $p\in \mathbb{P}$. In particular, $\mathrm{Spc}(\mathrm{D}(\mathfrak{C}^{\mathrm{pr}})^{c})\simeq \mathbb{P}^{*}$, the one-point compactification of $\mathbb{P}$.
\end{prop}
\begin{proof}For different primes $p$  and  $q$, we have $e_{p}\otimes e_{q}=0$. It follows that a prime ideal $Q$ in  $\mathrm{D}(\mathfrak{C}^{\mathrm{pr}})^{c}$ can exclude at most one $e_{p}$. Suppose $e_{p}\notin Q$ for some $p\in \mathbb{P}$. Since $Q$ is prime, the equality $e_{p}\otimes \gamma_{p}=0$ implies that $\gamma_{p}\in Q$. Thus $P_{p}\subseteq Q$ by Lemma \ref{lem3.9}. By \cite[Corollary 6.4]{BBP+25b}, $P_{p}$ is maximal, hence $Q=P_{p}$.

Now suppose $e_{p}\in Q$ for each $p\in \mathbb{P}$. Then, by the same argument, we have $P_{1}\subseteq Q$, which implies that $Q=P_{1}$. This proves the first part.

 It remains to determine the topology on $\mathrm{Spc}(\mathrm{D}(\mathfrak{C}^{\mathrm{pr}})^{c})$. Note that $\mathrm{supp}(e_{p})=\{P_{p}\}$ for each $p\in \mathbb{P}$. Thus each $P_{p}$ is closed. On the other hand $\mathrm{supp}(\gamma_{p})=\mathrm{Spc}(\mathrm{D}(\mathfrak{C}^{\mathrm{pr}})^{c})\setminus \{P_{p}\}$, so each $P_{p}$ is also open, hence isolated. For a finite subset $F\subseteq \mathbb{P}$, we have $\mathrm{supp}(\gamma_{F})=\{P_{1}\}\bigcup \{P_{p}\mid p\in F\}$. Thus every neighbourhood of $P_{1}$ contains all but finitely many of the isolated primes $P_{p}$. Conversely, if $X\in \mathrm{D}(\mathfrak{C}^{\mathrm{pr}})^{c}$ with $P_{1}\notin \mathrm{supp}(X)$, then $X(1)\simeq 0$ and so $\mathrm{supp}(X)$ is finite. Hence every closed subset not containing $P_{1}$ is finite. This proves the final statement.
\end{proof}

\begin{thm}\label{thm3.11}Every localizing ideal of $\mathrm{D}(\mathfrak{C}^{\mathrm{pr}})$ is torsion and the homological support induces  bijections  \[
\{\text{localizing ideals of } \mathrm{D}(\mathfrak{C}^{\mathrm{pr}})\}
\simeq
\{\text{subsets of } \mathbb{P}^{*}\}
\simeq
\{\text{subsets of } \mathrm{Spc}(\mathrm{D}(\mathfrak{C}^{\mathrm{pr}})^{c}\}.
\]

\end{thm}
\begin{proof} It suffices to prove that for any $X\in \mathrm{D}(\mathfrak{C}^{\mathrm{pr}})$, we have $X\in \mathcal{L}$ if and only if $\mathrm{hsupp}(X)\subseteq S(\mathcal{L})$. If $X\simeq 0$, this is clear. Now suppose $X\in \mathcal{L}$ and $p\in \mathrm{hsupp}(X)\subseteq \{1\}\bigcup \mathbb{P}$. Since $p\in \mathrm{hsupp}(X)$, we have $X(p)\not\simeq 0$, and so $\chi_{p,X(p)}\not\simeq 0$. Moreover, $$\chi_{p,k}\in \mathrm{Loc}_{\otimes}\langle \chi_{p,X(p)}\rangle\subseteq \mathrm{Loc}_{\otimes}\langle X\rangle\subseteq \mathcal{L},$$ where the first inclusion follows from Lemma \ref{lem3.1} and the second follows from the identity $X\otimes \chi_{p,k}\cong \chi_{p,X(p)}$.

 Conversely, let $X\in \mathrm{D}(\mathfrak{C}^{\mathrm{pr}})$ with $\mathrm{hsupp}(X)\subseteq S(\mathcal{L})$. Define the object $r_{\mathcal{L}}\in \mathrm{D}(\mathfrak{C}^{\mathrm{pr}})$ by \[
r_{\mathcal{L}}(q) = \begin{cases}
k & \text{if } q\in S(\mathcal{L})\\
0 & \text{if } q\notin S(\mathcal{L})
\end{cases}
\]
with the transition morphism $r_{\mathcal{L}}(1)\rightarrow r_{\mathcal{L}}(q)$ equal to identity if $1,q\in S(\mathcal{L})$ and zero otherwise. Then we have $r_{\mathcal{L}}\otimes r_{\mathcal{L}}\cong r_{\mathcal{L}}$ and $X\cong X\otimes r_{\mathcal{L}}$. Moreover, we have $r_{\mathcal{L}}\in \mathrm{Loc}_{\otimes}\langle \chi_{p,k}\mid p\in S(\mathcal{L})\rangle$. Indeed, if $1\notin S(\mathcal{L})$, then $r_{\mathcal{L}}\cong \bigoplus_{p\in S(\mathcal{L})}\chi_{p,k}$, which implies that $r_{\mathcal{L}}\in \mathrm{Loc}_{\otimes}\langle \chi_{p,k}\mid p\in S(\mathcal{L})\rangle$. If $1\in S(\mathcal{L})$, then there is a short exact sequence \[
\xymatrix@C=1cm{
  0 \ar[r] &
  \displaystyle\bigoplus_{\substack{p\in S(\mathcal{L}) \\ p\neq 1}} \chi_{p,k}
  \ar[r] &
  r_{\mathcal{L}} \ar[r] &
  \chi_{1,k} \ar[r] &
  0
}
\]
in $\mathrm{A}(\mathfrak{C}^{\mathrm{pr}})$.  The distinguished triangle in $\mathrm{D}(\mathfrak{C}^{\mathrm{pr}})$ induced by this sequence again shows that $r_{\mathcal{L}}\in \mathrm{Loc}_{\otimes}\langle \chi_{p,k}\mid p\in S(\mathcal{L})\rangle$. Therefore, $$X\in \mathrm{Loc}_{\otimes}\langle r_{\mathcal{L}}\rangle \subseteq \mathrm{Loc}_{\otimes}\langle \chi_{p,k}\mid p\in S(\mathcal{L})\rangle,$$ which concludes the proof.
\end{proof}

\begin{lem}\label{lem3.12}Let $\mathcal{L}_{S}=\mathrm{Loc}_{\otimes}\langle\chi_{p,k}\mid p\in S\rangle$, where $S=\{1\}\bigcup F$ for some subset $F\subseteq \mathbb{P}$, and let $Y\in \mathrm{D}(\mathfrak{C}^{\mathrm{pr}})$. Then $Y\in \mathcal{L}_{S}^\bot$ if and only if the following two conditions hold:\\
(1) $Y(p)=0$ for all $p\in F$; and\\
(2) the morphism
\[
Y(1)\longrightarrow \prod_{q\in \mathbb{P}\setminus F} Y(q)^{\Gamma_q}
\]
is a quasi-isomorphism, where the map is induced by the structure maps $Y(1)\to Y(q)$, with $Y(1)$ equipped with the trivial $\Gamma_q$-action.

\end{lem}
\begin{proof}  Consider the short exact sequence \[
\xymatrix@C=1cm{
  0 \ar[r] &
  \displaystyle\bigoplus_{\substack{p\in \mathbb{P} }} \chi_{p,k}
  \ar[r] &
  e_{1} \ar[r] &
  \chi_{1}\ar[r] &
  0
}
\]
in $\mathrm{A}(\mathfrak{C}^{\mathrm{pr}})$. This induces a distinguished triangle \[
\xymatrix@C=1cm{
  \displaystyle\bigoplus_{\substack{p\in \mathbb{P} }} \chi_{p,k}
  \ar[r] &
  e_{1} \ar[r] &
  \chi_{1}\ar[r] &
  \Sigma\displaystyle\bigoplus_{\substack{p\in \mathbb{P} }} \chi_{p,k}
}
\]
in $\mathrm{D}(\mathfrak{C}^{\mathrm{pr}})$. Applying the functor $\mathrm{Hom}_{\mathrm{D}(\mathfrak{C}^{\mathrm{pr}})}(-, Y)$ to this triangle we obtain a  long exact sequence
\[
\xymatrix@C=0.7cm@R=0.5cm{
  ~ &
  \mathrm{Hom}_{\mathrm{D}(\mathfrak{C}^{\mathrm{pr}})}
    \left(\bigoplus_{\substack{p\in \mathbb{P}}}\chi_{p,k},\, \Sigma^{n-1}Y\right)
  \ar[r] &
  \mathrm{Hom}_{\mathrm{D}(\mathfrak{C}^{\mathrm{pr}})}(\chi_1,\, \Sigma^nY)
  \ar[d] & & ~ \\
  & &
  \mathrm{Hom}_{\mathrm{D}(\mathfrak{C}^{\mathrm{pr}})}(e_1,\, \Sigma^nY)
  \ar[r]^-{\delta_n} &
  \mathrm{Hom}_{\mathrm{D}(\mathfrak{C}^{\mathrm{pr}})}
    \left(\bigoplus_{\substack{p\in \mathbb{P}}} \chi_{p,k},\, \Sigma^nY\right).
   &
}
\]
Note that $\mathrm{Hom}_{\mathrm{D}(\mathfrak{C}^{\mathrm{pr}})}(e_{1},\Sigma^{n}Y)=H_{n}(Y(1))$.
Since $$\mathrm{Hom}_{\mathrm{D}(\mathfrak{C}^{\mathrm{pr}})}(\chi_{p,k},\Sigma^{n}Y)=H_{n}(Y(p)^{\Gamma_{p}}),$$ we have
$$\mathrm{Hom}_{\mathrm{D}(\mathfrak{C}^{\mathrm{pr}})}(\displaystyle\bigoplus_{\substack{p\in \mathbb{P} }} \chi_{p,k},\Sigma^{n}Y)=\displaystyle\prod_{\substack{p\in \mathbb{P} }}H_{n}(Y(p)^{\Gamma_{p}}).$$ Under these identifications, $\delta_{n}$ is induced by the structural map $\delta: Y(1)\rightarrow \displaystyle\prod_{\substack{p\in \mathbb{P} }}Y(p)^{\Gamma_{p}}$. Therefore, $\mathrm{Hom}_{\mathrm{D}(\mathfrak{C}^{\mathrm{pr}})}(\chi_1,\, \Sigma^nY)=0$ for $\forall n\in \mathbb{Z}$  if and only if $\delta: Y(1)\rightarrow \displaystyle\prod_{\substack{p\in \mathbb{P} }}Y(p)^{\Gamma_{p}}$ is a quasi-isomorphism.

Assume $Y\in \mathcal{L}_{S}^\bot$. Since $e_{p}=\chi_{p}\in \mathcal{L}_{S}$ for $p\in F$, we have $$H_{n}(Y(p))=\mathrm{Hom}_{\mathrm{D}(\mathfrak{C}^{\mathrm{pr}})}(\Sigma^{n}e_{p}, Y)=\mathrm{Hom}_{\mathrm{D}(\mathfrak{C}^{\mathrm{pr}})}(\Sigma^{n}\chi_{p}, Y)=0$$ for all $ n\in \mathbb{Z}, p\in F$. Thus $Y(p)=0$ for $\forall p\in F$, so $Y$ satisfies condition (1). The above argument, together with  condition (1), implies condition (2).

The converse implication is similar.
\end{proof}

The following result characterizes when a localizing ideal $\mathcal{L}$ satisfying $1\in S(\mathcal{L})$ is smashing.
\begin{prop}\label{prop3.13}Let $\mathcal{L}_{S}=\mathrm{Loc}_{\otimes}\langle\chi_{p,k}\mid p\in S\rangle$, where $S=\{1\}\bigcup F$ for some subset $F\subseteq \mathbb{P}$. Then $\mathcal{L}_{S}^{\bot}$ is localizing if and only if $\mathbb{P}\setminus F$ is finite.
\end{prop}
\begin{proof}Assume $\mathbb{P}\setminus F$ is finite. Let $j: V_{F}:=\{C_{q}\mid q\in \mathbb{P}\setminus F\}\hookrightarrow \mathfrak{C}^{\mathrm{pr}}$ denote the inclusion of the subcategory. The right Kan extension along $j$ induces a functor $j_{*}: \mathrm{D}(V_{F})\rightarrow \mathrm{D}(\mathfrak{C}^{\mathrm{pr}})$. It is easy to check that for any $X\in \mathrm{D}(V_{F})$, we have \[
(j_{*}X)(r) = \begin{cases}
X(r) & \text{if } r\in \mathbb{P}\setminus F\\
0 & \text{if } r\in F\\
\displaystyle\prod_{\substack{q\in \mathbb{P}\setminus F }}X(q)^{\Gamma_{q}} & r=1.
\end{cases}
\]
Since finite products, invariants under finite groups and direct sums all commute with each other, $j_{*}$ preserves coproducts. Hence the essential image $\mathrm{Im}(j_{*})$  is a localizing subcategory of $\mathrm{D}(\mathfrak{C}^{\mathrm{pr}})$. To prove that $\mathcal{L}_{S}^{\bot}$ is a localizing subcategory, it suffices to show that $\mathcal{L}_{S}^{\bot}=\mathrm{Im}(j_{*})$. The "$\supseteq$" containment follows from Lemma \ref{lem3.12}.\\
 For the reverse containment, we claim that the unit morphism $\eta_{Y}:Y\rightarrow j_{*}j^{*}Y$, induced by the adjunction, is an isomorphism for each $Y\in \mathcal{L}_{S}^{\bot}$. Indeed, for each $Y\in \mathrm{D}(\mathfrak{C}^{\mathrm{pr}})$, $\eta_{Y}$ is given by  \[
\eta_{Y}(r) = \begin{cases}
\mathrm{Id}_{Y(r)} & \text{if } r\in \mathbb{P}\setminus F\\
Y(r)\rightarrow 0 & \text{if } r\in F\\
Y(1)\rightarrow \displaystyle\prod_{\substack{q\in \mathbb{P}\setminus F }}Y(q)^{\Gamma_{q}} & r=1.
\end{cases}
\]
 It follows from Lemma \ref{lem3.12} that $\eta_{Y}$ is an isomorphism  for $Y\in \mathcal{L}_{S}^{\bot}$.\\
  Now suppose $\mathcal{L}_{S}^{\bot}$ is localizing. Assume $\mathbb{P}\setminus F$ is infinite. For each $q\in \mathbb{P}\setminus F$, define the object $Y_{q}$ by $Y_{q}(r)=k$ if $r=1,q$ and zero otherwise, with the structural map $Y_{q}(1)\rightarrow Y_{q}(q)$ equal to the identity. Then $Y_{q}\in \mathcal{L}_{S}^{\bot}$ for each $q\in \mathbb{P}\setminus F$, again by Lemma \ref{lem3.12}. Take the coproduct $Y:=\displaystyle\bigoplus_{\substack{q\in \mathbb{P}\setminus F }} Y_{q}$, which belongs to $\mathcal{L}_{S}^{\bot}$. Then we have a quasi-isomorphism $$Y(1)\simeq\displaystyle\bigoplus_{\substack{q\in \mathbb{P}\setminus F }} k \rightarrow \displaystyle\prod_{\substack{q\in \mathbb{P}\setminus F }}Y(q)^{\Gamma_{q}}\simeq\displaystyle\prod_{\substack{q\in \mathbb{P}\setminus F }}k,$$ which is impossible, since the left-hand side is an infinite direct sum and the right-hand side is an infinite direct product. This contradiction completes the proof.
\end{proof}

We are now ready to state the main theorem of this section.
\begin{thm}\label{thm3.14} The telescope conjecture holds for $\mathrm{D}(\mathfrak{C}^{\mathrm{pr}})$. Moreover, the homological support induces  bijections \[
\{\text{smashing ideals of } \mathrm{D}(\mathfrak{C}^{\mathrm{pr}})\}
\simeq
\{\text{open subsets of } \mathbb{P}^{*}\}
\simeq
\{\text{open subsets of } \mathrm{Spc}(\mathrm{D}(\mathfrak{C}^{\mathrm{pr}})^{c}\}.
\]
\end{thm}
\begin{proof} By Theorem \ref{thm3.11}, each localizing ideal in $\mathrm{D}(\mathfrak{C}^{\mathrm{pr}})$ is of the form $\mathcal{L}=\mathrm{Loc}_{\otimes}\langle \chi_{p,k}\mid p\in S(\mathcal{L})\rangle$. If $1$ is not in $S(\mathcal{L})$, then $\mathcal{L}$ is compactly generated since $\chi_{p,k}=e_{p,k}$ is compact for $p\neq 1$. Now suppose $1\in S(\mathcal{L})$. We write  $S(\mathcal{L})=\{1\}\bigcup F$ with $F\subseteq \mathbb{P}$. By Proposition \ref{prop3.13}, $\mathcal{L}=\mathcal{L}_{S(\mathcal{L})}$ is smashing if and only if $\mathbb{P}\setminus F$ is finite. Now suppose $\mathbb{P}\setminus F$ is finite. We show that $\mathcal{L}$ is compactly generated.\\
Recall in the proof of Lemma \ref{lem3.9} that there is a triangle  $$\bigoplus_{p\in \mathbb{P}\setminus F}e_{p,k}\rightarrow e_{1}\rightarrow \gamma_{\mathbb{P}\setminus F}$$ in $\mathrm{D}(\mathfrak{C}^{\mathrm{pr}})$ with both  $\bigoplus_{p\in \mathbb{P}\setminus F}e_{p,k}$ and $e_{1}$ compact, it follows that $\gamma_{\mathbb{P}\setminus F}$ is also compact. There is another triangle $$\bigoplus_{p\in F}\chi_{p,k}\rightarrow \gamma_{\mathbb{P}\setminus F}\rightarrow \chi_{1,k},$$ which shows that $\chi_{1,k}\in \mathrm{Loc}_{\otimes}\langle \gamma_{\mathbb{P}\setminus F}, \chi_{p,k}\mid p\in F\rangle$. Hence we have $$\mathcal{L}=\mathrm{Loc}_{\otimes}\langle \chi_{1,k}, \chi_{p,k}\mid p\in F \rangle=\mathrm{Loc}_{\otimes}\langle \gamma_{\mathbb{P}\setminus F}, e_{p,k}\mid p\in F \rangle$$ which is compactly generated. This concludes our proof.

\end{proof}

\section{\bf Global representations of  $\mathfrak{C}_{p}$}

Let $\mathfrak{C}_{p}$ denote the family of cyclic $p$-groups for a fixed prime number $p$. In this section, we study the derived category $\mathrm{D}(\mathfrak{C}_{p})$ of $\mathfrak{C}_{p}$-global representations. It is shown in \cite{PS22} that  $\mathrm{A}(\mathfrak{C}_{p})$ is a locally noetherian Grothendieck abelian category. Let $\mathrm{A}(\mathfrak{C}_{p})^{c}$ denote the abelian subcategory of finitely generated objects. We show that there is a tensor equivalence between $\mathrm{D}(\mathfrak{C}_{p})^{c}$ and the bounded derived category of $\mathrm{A}(\mathfrak{C}_{p})^{c}$. We then provide a new  method for classifying  prime ideals of $\mathrm{D}(\mathfrak{C}_{p})^{c}$, using the tensor abelian geometry of $\mathrm{A}(\mathfrak{C}_{p})^{c}$ developed in \cite{Xu26}. Finally we classify the localizing ideals of $\mathrm{D}(\mathfrak{C}_{p})$ and prove that the telescope conjecture holds for $\mathrm{D}(\mathfrak{C}_{p})$.

\begin{df} We say $\mathcal{U}$ is noetherian if $\mathrm{A}(\mathcal{U})$ is locally noetherian, i.e., every subobject of $e_{G}$ is finitely generated for each $G\in \mathcal{U}$. We denote by $\mathrm{A}(\mathcal{U})^{c}$ the  abelian subcategory of finitely generated objects. We write $\mathrm{D}^{b}(\mathrm{A}(\mathcal{U})^{c})$ for the bounded derived category of $\mathrm{A}(\mathcal{U})^{c}$ and write \[
\operatorname{D}^{b}_{\mathrm{fg}}(\mathcal{U})
:=
\left\{
X\in \mathrm{D}(\mathcal{U})
\;\middle|\;
H_i(X)\in \mathrm{A}(\mathcal{U})^{c}\ \forall i\in\mathbb{Z},\;
H_n(X)=0\ \text{for } |n|\gg 0
\right\}.
\]
\end{df}

\begin{prop}\label{prop4.2}Let $\mathcal{U}$ be noetherian and satisfy Hypothesis \ref{hyp2.3}. Then we have the following:
\begin{enumerate}
\item $\mathrm{A}(\mathcal{U})^{c}$ is a tensor abelian category;

\item $\operatorname{D}^{b}_{\mathrm{fg}}(\mathcal{U})$ is a tensor triangulated category;

\item The inclusion $\mathrm{D}^{b}(\mathrm{A}(\mathcal{U})^{c})\hookrightarrow \operatorname{D}^{b}_{\mathrm{fg}}(\mathcal{U})$ is a tensor equivalence;

\item $\mathrm{D}(\mathcal{U})^{c}$ is a tensor triangulated subcategory of $\operatorname{D}^{b}_{\mathrm{fg}}(\mathcal{U})$.
\end{enumerate}
\end{prop}
\begin{proof}  By the assumptions and \cite[Theorem 7.3]{BBP+25a}, the tensor unit $1$ is finitely generated. Then (1) holds by \cite[Proposition 8.7]{PS22}. For part (2), by (1), the tensor product of two finitely generated objects is again finitely generated. Moreover, since every object in $\mathrm{A}(\mathcal{U})$ is flat, the K\"{u}nneth formula applies, so the tensor product of two objects in $\operatorname{D}^{b}_{\mathrm{fg}}(\mathcal{U})$ again lies in $\operatorname{D}^{b}_{\mathrm{fg}}(\mathcal{U})$. This establishes (2). Part (3) comes from a standard comparison theorem for derived categories of Serre subcategories, see, for example, \cite[Tag 0FCL]{S24}. Part (4) follows directly from the definitions.
\end{proof}

\begin{ex}\label{ex4.3} $\mathfrak{C}^{\mathrm{pr}}$ is not noetherian, by \cite[Proposition 13.3]{PS22}.
\end{ex}

\begin{ex}\label{ex4.4}If $\mathcal{U}$ is essentially finite, then $\mathcal{U}$ is noetherian by \cite[Proposition 11.9]{PS22} and $\mathrm{D}(\mathcal{U})^{c}=\operatorname{D}^{b}_{\mathrm{fg}}(\mathcal{U})$ by \cite[Proposition 7.4]{BBP+25a}.
\end{ex}

\begin{ex}\label{ex4.5}Let $E_{p}$ denote the family of elementary abelian $p$-groups. Then $E_{p}$ is noetherian by \cite[Theorem 13.4]{PS22}. We also have $\mathrm{D}(E_{p})^{c}\subsetneq \operatorname{D}^{b}_{\mathrm{fg}}(E_{p})$. Indeed, for each $G\in E_{p}$, $\chi_{G,k}\in \operatorname{D}^{b}_{\mathrm{fg}}(E_{p})$ but $\chi_{G,k}$ is not compact by \cite[Theorem 7.7]{BBP+25a}.
\end{ex}

\begin{df}Let $\mathcal{U}$ be noetherian and satisfy Hypothesis \ref{hyp2.3}. For $G\in \mathcal{U}$, we define the subcategory $P^{b}_{G}\subset \operatorname{D}^{b}_{\mathrm{fg}}(\mathcal{U})$ by
$$P^{b}_{G}:=\{X\in\operatorname{D}^{b}_{\mathrm{fg}}(\mathcal{U}) \mid H_{*}(X)(G)=0          \}.$$
We call such a subcategory a group prime.
\end{df}

\begin{lem}\label{lem4.7}Let $\mathcal{U}$ be noetherian and satisfy Hypothesis \ref{hyp2.3}. Then $(\pi_{0}(\mathcal{U}), \mathrm{hsupp})$ is a support datum on  $\operatorname{D}^{b}_{\mathrm{fg}}(\mathcal{U})$, which induces a continuous injective map $$P^{b}: \pi_{0}(\mathcal{U})\rightarrow \mathrm{Spc}(\operatorname{D}^{b}_{\mathrm{fg}}(\mathcal{U})),\quad\quad [G]\mapsto P^{b}_{G}.$$
\end{lem}
\begin{proof}The proof is the same as in \cite[Proposition 4.14]{BBP+25b}.
\end{proof}

Let $P\subseteq \mathrm{A}(\mathcal{U})^{c}$ be a prime Serre ideal, in the sense of \cite[Definition 3.1]{Xu26}. We have a subcategory $\Phi(P):=\{X\in \operatorname{D}^{b}_{\mathrm{fg}}(\mathcal{U})\mid H_{i}(X)\in P,~\forall i\in \mathbb{Z}\}$ of $\operatorname{D}^{b}_{\mathrm{fg}}(\mathcal{U})$. The following result shows that it is actually a prime ideal.
\begin{prop}\label{prop4.8} Let $\mathcal{U}$ be noetherian and satisfy Hypothesis \ref{hyp2.3}. Then there is a continuous injective map $$\Phi: \mathrm{Spc}(\mathrm{A}(\mathcal{U})^{c})\rightarrow \mathrm{Spc}(\operatorname{D}^{b}_{\mathrm{fg}}(\mathcal{U})),\quad\quad P\mapsto \Phi(P).$$
\end{prop}
\begin{proof}First, let $P\subseteq \mathrm{A}(\mathcal{U})^{c}$ be a prime Serre ideal. We first show that $\Phi(P)$ is a prime thick ideal of $\operatorname{D}^{b}_{\mathrm{fg}}(\mathcal{U})$. Note that $\mathrm{D}^{b}(\mathrm{A}(\mathcal{U})^{c})\simeq\operatorname{D}^{b}_{\mathrm{fg}}(\mathcal{U})$. It is easy to see that $\Phi(P)$ is a thick ideal. Suppose $X,Y\notin \Phi(P)$. Choose the smallest degrees $m$ and $n$ such that $H_{m}(X)\notin P$ and $H_{n}(Y)\notin P$. Consider the functor $\pi: \mathrm{D}^{b}(\mathrm{A}(\mathcal{U})^{c})\rightarrow \mathrm{D}^{b}(\mathrm{A}(\mathcal{U})^{c}/P)$. Then $H_{m}(\pi(X))\neq 0$ and $H_{n}(\pi(Y))\neq 0$ with all lower homology objects vanishing. Hence \[
\begin{aligned}
    \pi (H_{m+n}(X\otimes Y)) & \simeq H_{m+n}(\pi (X\otimes Y) )\\
         & \simeq H_{m+n}(\pi (X)\otimes \pi (Y) ) \\
         & \simeq H_{m}(\pi (X))\otimes H_{n}(\pi (Y))\\
         &\neq 0,
\end{aligned}
\]
which implies that $H_{m+n}(X\otimes Y)\notin P$ and thus $X\otimes Y\notin \Phi(P)$.\\
The injectivity of $\Phi$ follows from the identity $\Phi(P)\bigcap \mathrm{A}(\mathcal{U})^{c}= P$, for each $P\in \mathrm{Spc}(\mathrm{A}(\mathcal{U})^{c})$. Now let $X\in \operatorname{D}^{b}_{\mathrm{fg}}(\mathcal{U})$. Note that $X\notin \Phi(P)$ if and only if there exists some $n$ such that $H_{n}(X)\notin P$.   Then $$\Phi^{-1}\mathrm{supp}(X)  = \{P\in \mathrm{Spc}(\mathrm{A}(\mathcal{U})^{c})\mid X\notin \Phi(P)   \}=\bigcup_{n} \mathrm{supp}(H_{n}(X)),$$ which is closed since $X$ is a bounded complex. This proves that $\Phi$ is continuous and concludes the proof.
\end{proof}

\begin{rem} Let $\mathcal{U}$ be noetherian and satisfy Hypothesis \ref{hyp2.3}. Then $(\pi_{0}(\mathcal{U}), \mathrm{hsupp})$ is also a support datum on $\mathrm{A}(\mathcal{U})^{c}$ and $\mathrm{D}(\mathcal{U})^{c}$, by \cite[Lemma 4.6]{Xu26} and \cite[Lemma 4.8]{BBP+25b}. Let $P^{a}$ and $P^{c}$ denote the induced maps from $\pi_0(\mathcal{U})$ to the spectrum of the tensor abelian category $\mathrm{A}(\mathcal{U})^{c}$ and to the Balmer spectrum of compact objects, respectively. Then we have the following diagram \[
\xymatrix{
  & & \pi_{0}(\mathcal{U}) \ar[d]^{P^{b}}\ar[dl]_{P^{a}}\ar[dr]^{P^{c}} & \\
  & \mathrm{Spc}(\mathrm{A}(\mathcal{U})^{c}) \ar[r]_{\Phi} & \mathrm{Spc}(\operatorname{D}^{b}_{\mathrm{fg}}(\mathcal{U})) \ar[r]_{\Psi} & \mathrm{Spc}(\mathrm{D}(\mathcal{U})^{c}),
}
\]
where $\Psi$ is the continuous map induced by the inclusion $\mathrm{D}(\mathcal{U})^{c}\hookrightarrow \operatorname{D}^{b}_{\mathrm{fg}}(\mathcal{U})$. This diagram connects the two spaces  $\mathrm{Spc}(\mathrm{A}(\mathcal{U})^{c})$ and $\mathrm{Spc}(\mathrm{D}(\mathcal{U})^{c})$.  We will see in Proposition \ref{prop4.11} that the composition $\Psi\circ\Phi$ is a homeomorphism when  $\mathcal{U}=\mathfrak{C}_{p}$.
\end{rem}

\begin{nota} In the rest of this section, we study the global representations of $\mathfrak{C}_{p}$. For simplicity, we write $e_{n,V}$ for $e_{C_{p^{n}},V}$,~$X(n)$ for $X(C_{p^{n}})$,  $\Gamma_{n}$ for $\mathrm{Out}(C_{p^{n}})=(\mathbb{Z}/p^{n})^{\times}$ and so on. We set $$T^{D}:=\mathrm{Loc}_{\otimes}\langle \chi_{i,k}\mid i\in \mathbb{N}\rangle\subseteq \mathrm{D}(\mathfrak{C}_{p})$$ and $$T^{A}:=\mathrm{Loc}_{\otimes}\langle \chi_{i,k}\mid i\in \mathbb{N}\rangle\subseteq \mathrm{A}(\mathfrak{C}_{p}).$$
\end{nota}
Recall from \cite[Theorem 5.7]{Xu26} that the only prime Serre ideals of $\mathrm{A}(\mathfrak{C}_{p})^{c}$ are $P_{n}:=\{X\in \mathrm{A}(\mathfrak{C}_{p})^{c}\mid X(n)=0\}$  for $n\in \mathbb{N}$ and $P_{\infty}:=\{X\in \mathrm{A}(\mathfrak{C}_{p})^{c}\mid \operatorname{hsupp}(X)\text{ is finite}\}$ .

\begin{prop}\label{prop4.11} We have $\mathrm{D}(\mathfrak{C}_{p})^{c}=\operatorname{D}^{b}_{\mathrm{fg}}(\mathfrak{C}_{p})$, and $\Phi:\mathrm{Spc}(\mathrm{A}(\mathfrak{C}_{p})^{c})\rightarrow \mathrm{Spc}(\operatorname{D}^{b}_{\mathrm{fg}}(\mathfrak{C}_{p}))$ is a homeomorphism. In particular, $\mathrm{Spc}(\mathrm{D}(\mathfrak{C}_{p})^{c})\simeq \mathbb{N}^{*}$.
\end{prop}
\begin{proof}  By Proposition \ref{prop4.2} we have $\mathrm{D}(\mathfrak{C}_{p})^{c}\subseteq\operatorname{D}^{b}_{\mathrm{fg}}(\mathfrak{C}_{p})$. Now we prove the reverse containment. By the bounded truncation triangle, it suffices to show that each object in $\mathrm{A}(\mathfrak{C}_{p})^{c}$ is compact when viewed as a complex concentrated in degree 0. Note that  \[
e_{n,V}(m) = \begin{cases}
\mathrm{Inf}^{\Gamma_{m}}_{\Gamma_{n}}V & \text{if } m\geq n\\
0 & \text{if } m<n
\end{cases}
\]
and there exists a short exact sequence $$\xymatrix@C=0.5cm{
  0 \ar[r] & e_{n+1,k} \ar[r] & e_{n,k} \ar[r] &\chi_{n,k} \ar[r] & 0 }$$
in $\mathrm{A}(\mathfrak{C}_{p})^{c}$. Thus $\chi_{n,k}\in \mathrm{D}(\mathfrak{C}_{p})^{c}$ for each $n\in \mathbb{N}$. Now let $M\in \mathrm{A}(\mathfrak{C}_{p})^{c}$. Choose a finite presentation of $M$ $$\xymatrix@C=0.5cm{
  \displaystyle\bigoplus_{\substack{\alpha\in I }} e_{a_{\alpha},V_{\alpha}} \ar[r] & \displaystyle\bigoplus_{\substack{\beta\in J }} e_{b_{\beta},V_{\beta}} \ar[r] & M \ar[r] & 0 }$$
with $I$ and $J$ finite. Fix some $N$ larger than all $a_{\alpha}$ and $b_{\beta}$. Then for each $m>N$, we have $M(m)\cong \mathrm{Inf}^{\Gamma_{m}}_{\Gamma_{N}}M(N)$. So the natural map $f:e_{N,M(N)}\rightarrow M$ is an isomorphism for all sufficiently large degrees. Moreover, $\mathrm{Ker}(f)$ and $\mathrm{Coker}(f)$ are supported on a finite set of groups. Thus the objects $e_{N,M(N)}$, $\mathrm{Ker}(f)$ and $\mathrm{Coker}(f)$ are all compact, which implies that $M$ is compact.\\
 We  now prove  that the only prime ideals in $\operatorname{D}^{b}_{\mathrm{fg}}(\mathfrak{C}_{p})$ are $\mathcal{P}_{n}:=\Phi(P_{n})$ for $n\in \mathbb{N}$ and $\mathcal{P}_{\infty}:=\Phi(P_{\infty})$. Let $\mathcal{P}\subseteq \operatorname{D}^{b}_{\mathrm{fg}}(\mathfrak{C}_{p})$ be a prime ideal. Since $\chi_{i}\otimes \chi_{j}=0$ when $i\neq j$,  at most one $\chi_{i}$ can lie outside $\mathcal{P}$. \\
 Suppose $\chi_{i}\notin \mathcal{P}$ for some $i$. Then the identity $\chi_{i}\otimes e_{i+1,k}=0$ implies $e_{i+1,k}\in \mathcal{P}$. We claim that $\mathcal{P}_{i}\subseteq \mathrm{thick}_{\otimes}\langle e_{i+1,k}, \chi_{j}\mid j\in \mathbb{N}, j\neq i \rangle$. It suffices to prove that every $M\in \mathrm{A}(\mathfrak{C}_{p})^{c}$ with $M(i)=0$ belongs to $\mathrm{thick}_{\otimes}\langle e_{i+1,k}, \chi_{j}\mid j\in \mathbb{N}, j\neq i \rangle$. By the argument above, there exists some $N\gg i$ such that the natural map $f: e_{N,M(N)}\rightarrow M$ is an isomorphism for all sufficiently large degrees $m>N$. Since $\mathrm{hsupp}(\mathrm{Ker}(f))$ and $\mathrm{hsupp}(\mathrm{Coker}(f))$ are finite and do not contain $i$, it follows that $\mathrm{Ker}(f)$ and $\mathrm{Coker}(f)$ are contained in $\mathrm{thick}_{\otimes}\langle e_{i+1,k}, \chi_{j}\mid j\in \mathbb{N}, j\neq i \rangle$. We then have $M\in \mathrm{thick}_{\otimes}\langle e_{i+1,k}, \chi_{j}\mid j\in \mathbb{N}, j\neq i \rangle$  since the same holds for $e_{N,M(N)}$. This proves the claim. Now we have $$\mathcal{P}_{i}\subseteq \mathrm{thick}_{\otimes}\langle e_{i+1,k}, \chi_{j}\mid j\in \mathbb{N}, j\neq i \rangle\subseteq \mathcal{P}.$$ But $\mathcal{P}_{i}$ is maximal, since $\operatorname{D}^{b}_{\mathrm{fg}}(\mathfrak{C}_{p})/\mathcal{P}_{i}\simeq \mathrm{D}^{b}(k[\Gamma_{i}])$. Therefore we have $\mathcal{P}=\mathcal{P}_{i}$.\\
 Suppose $\chi_{i}\in \mathcal{P}$ for every $i\in \mathbb{N}$. Then $\mathcal{P}_{\infty}=\mathrm{thick}_{\otimes}\langle \chi_{i}\mid i\in \mathbb{N}\rangle\subseteq \mathcal{P}$. But $\mathcal{P}_{\infty}$ is  maximal since $\operatorname{D}^{b}_{\mathrm{fg}}(\mathfrak{C}_{p})/\mathcal{P}_{\infty}\simeq \mathrm{D}^{b}(\mathrm{A}(\mathfrak{C}_{p})^{c}/P_{\infty}$) is equivalent to the bounded derived category of a semisimple tensor abelian category. Thus we have $\mathcal{P}=\mathcal{P}_{\infty}$, and so $\mathcal{P}_{n}$ for $n\in \mathbb{N}$ and $\mathcal{P}_{\infty}$ are the only prime ideals in $\operatorname{D}^{b}_{\mathrm{fg}}(\mathfrak{C}_{p})$.\\
 It follows immediately that $\Phi$ is a homeomorphism.
\end{proof}

Finally, we study the localizing ideals of $\mathrm{D}(\mathfrak{C}_{p})$.
\begin{lem}\label{lem4.12} Let $\mathcal{A}$ be a locally noetherian Grothendieck tensor abelian category, with $\mathcal{A}^{c}$ the subcategory of noetherian objects. If $\mathcal{A}^{c}$ is closed under tensor product, there is a bijection
\[
\{\text{localizing ideals of }\mathcal{A}\}
\rightarrow
\{\text{Serre ideals of }\mathcal{A}^c\},\quad
\mathcal{L}\mapsto \mathcal{L}\cap \mathcal{A}^c.
\]
\end{lem}
\begin{proof}
This follows from the classical Gabriel correspondence. Indeed, under our assumption, the correspondence preserves tensor ideals, since both the forward and inverse maps are compatible with the tensor product. We leave the details to the reader.
\end{proof}

\begin{lem}\label{lem4.13} Let $X\in \mathrm{A}(\mathfrak{C}_{p})$. Then we have the following:
\begin{enumerate}
\item $X\in T^{A}$ if and only if $X$ is torsion;

\item If $X$ is not torsion, then $\mathrm{hsupp}(X)$ is cofinite.

\end{enumerate}
\end{lem}
\begin{proof}For part (1), the forward implication is clear. Conversely, suppose $X$ is torsion and write $X$ as a filtered colimit $\operatorname{colim}_i X_i$ of its finitely generated subobjects. By \cite[Lemma 5.5]{Xu26}, it suffices to prove $X_{i}\in P_{\infty}$ for each $i$. By the argument in the proof of Proposition \ref{prop4.11}, there exists some $r_{i}\in \mathbb{N}$  such that $e_{r_{i},X_{i}(r_{i})}(n)\cong X_{i}(n)$ for every $n\geq r_{i}$. This forces $X_{i}(n)=0$ for $n\geq r_{i}$, since $X_{i}$ is torsion.  This implies that $X_{i}\in P_{\infty}$.\\
Part (2) follows directly from the definition.
\end{proof}

\begin{df} Let $S\subseteq \mathbb{N}$ be a cofinite subset. We write $$\Delta_{S}:=\mathrm{min}\{n\in \mathbb{N}\mid \{m\in \mathbb{N}: m\geq n\}\subseteq S  \}.$$
\end{df}

\begin{thm}\label{thm4.15} Let $\mathcal{L}\subseteq \mathrm{D}(\mathfrak{C}_{p})$ be a localizing ideal. Then we have the following:
 \begin{enumerate}
\item  If $\mathcal{L}\subseteq T^{D}$, then $\mathcal{L}$ is torsion, i.e., $\mathcal{L}=\mathrm{Loc}_{\otimes}\langle  \chi_{i,k}\mid i \in S(\mathcal{L})\rangle$;

\item If $\mathcal{L}\nsubseteq T^{D}$ , then $S(\mathcal{L})$ is cofinite and $\mathcal{L}=\mathrm{Loc}_{\otimes}\langle e_{\Delta_{S(\mathcal{L})}}, \chi_{i,k}\mid i \in S(\mathcal{L})\rangle$.

\end{enumerate}
 In particular, there are canonical bijections between the following six classes:

\begin{enumerate}
\item[(a)] localizing ideals of $\mathrm{D}(\mathfrak{C}_{p})$;
\item[(b)] thick ideals of $\mathrm{D}(\mathfrak{C}_{p})^{c}$;
\item[(c)] open subsets of $\mathrm{Spc}(\mathrm{D}(\mathfrak{C}_{p})^{c})$;
\item[(d)] open subsets of $\mathbb{N}^{*}$;
\item[(e)] localizing ideals of $\mathrm{A}(\mathfrak{C}_{p})$;
\item[(f)] Serre ideals of $\mathrm{A}(\mathfrak{C}_{p})^{c}$.
\end{enumerate}
\end{thm}
\begin{proof}First we show that $T^{D}=\{X\in \mathrm{D}(\mathfrak{C}_{p})\mid H_{i}(X)\in T^{A}, \forall i\in \mathbb{Z}\}$. The $"\subseteq"$ containment is clear. For the reverse containment, first observe that every object of $T^{A}$, regarded as a complex concentrated in degree zero, belongs to $T^{D}$. Indeed, every torsion module $M$ is a filtered colimit of its finitely generated torsion submodules $M_{\lambda}$. Each finitely generated torsion module has finite support.  Applying Proposition \ref{prop3.2} to a finite truncation containing $\mathrm{hsupp}(M_{\lambda})$, and then extending by zero, we obtain  $$M_{\lambda}\in \mathrm{Loc}_{\otimes}\langle \chi_{i,k}\mid i\in \mathrm{hsupp}(M_{\lambda})\rangle\subseteq T^{D}.$$ Since $T^{D}$ is localizing, it follows that $M\in T^{D}$. Then by induction on $n$, each bounded truncation $\tau_{\leq n}\tau_{\geq-n}X$ lies in $T^{D}$. Since $T^{D}$ is localizing and $X$ is the homotopy colimit of $\tau_{\leq n}\tau_{\geq-n}X$, we have $X\in T^{D}$.

Suppose $\mathcal{L}\subseteq T^{D}$. It suffices to prove $X\in \mathrm{Loc}_{\otimes}\langle \chi_{i,k}\mid i\in \mathrm{hsupp}(X)\rangle$ for $X\in \mathcal{L}$. Indeed, if $X\in \mathcal{L}$, then $H_{i}(X)\in \mathrm{Loc}_{\otimes}\langle \chi_{i,k}\mid i\in \mathrm{hsupp}(X)\rangle$. By bounded truncation triangle, we have $\tau_{\leq n}\tau_{\geq-n}X\in \mathrm{Loc}_{\otimes}\langle \chi_{i,k}\mid i\in \mathrm{hsupp}(X)\rangle$.  Thus $X\in \mathrm{Loc}_{\otimes}\langle \chi_{i,k}\mid i\in \mathrm{hsupp}(X)\rangle$ since it is the homotopy colimit of $\tau_{\leq n}\tau_{\geq-n}X$. This implies that the assignment  $\mathcal{L}\mapsto S(\mathcal{L})$ gives a bijection between torsion localizing ideals of $\mathrm{D}(\mathfrak{C}_p)$ and subsets of $\mathbb{N}$.\\
Suppose now $\mathcal{L}\nsubseteq T^{D}$. Then there exists some $X\in \mathcal{L}$ such that $X\notin T^{D}$. By the derived Gabriel localization theorem, the quotient functor induces a tensor equivalence $$\mathrm{D}(\mathfrak{C}_{p})/T^{D}\simeq \mathrm{D}(\mathrm{A}(\mathfrak{C}_{p})/T^{A}).$$  There exists some $d\in \mathbb{Z}$ such that $H_{d}(X)\notin T^{A}$. By Lemma \ref{lem4.13}, $H_{d}(X)$ is not torsion and there exists some $r\in \mathbb{N}$ such that $H_{d}(X)(n)\neq 0$ for all $n\geq r$. Let $V_{n}$ denote the dual of  $H_{d}(X)(n)$. Then  $\chi_{n,k}$ is a direct summand of $H_{d}(X\otimes \chi_{n,V_{n}})$, since $$H_{d}(X\otimes \chi_{n,V_{n}})\cong H_{d}(X)\otimes \chi_{n,V_{n}}\cong \chi_{n,H_{d}(X)(n)\otimes V_{n}}.$$ On the other hand, $H_{d}(X\otimes \chi_{n,V_{n}})[-d]$ is a direct summand of $X\otimes \chi_{n,V_{n}}$, because $X\otimes \chi_{n,V_{n}}$ lies in $\mathrm{D}(\Gamma_{n})$, where every object is isomorphic to the direct sum of its homology objects since $k[\Gamma_n]$ is semisimple. Thus we have $\chi_{n,k}\in \mathrm{Loc}_{\otimes}\langle X\rangle\subseteq \mathcal{L}$ for all $n\geq r$.  Now we claim that $e_{r}\in \mathcal{L}$. Observe that each object in $\mathrm{A}(\mathfrak{C}_{p})$ is eventually constant with identities as transitions, $\mathrm{A}(\mathfrak{C}_{p})/T^{A}$ is a semisimple locally noetherian Grothendieck tensor abelian category, and for each simple object $\overline{M}\in \mathrm{A}(\mathfrak{C}_{p})/T^{A}$, the tensor unit $\overline{1}$ is a direct summand of $\overline{M\otimes M^{*}}$.   Thus the only localizing ideals of $\mathrm{D}(\mathrm{A}(\mathfrak{C}_{p})/T^{A})$ are $0$ and $\mathrm{D}(\mathrm{A}(\mathfrak{C}_{p})/T^{A})$. Consider the quotient functor $$\pi: \mathrm{D}(\mathrm{A}(\mathfrak{C}_{p}))\rightarrow \mathrm{D}(\mathrm{A}(\mathfrak{C}_{p}))/T^{D}\simeq \mathrm{D}(\mathrm{A}(\mathfrak{C}_{p})/T^{A}).$$
 Then we have $\pi^{-1}\mathrm{Loc}_{\otimes}\langle \pi(X)\rangle=\mathrm{Loc}_{\otimes}\langle X, T^{D}\rangle=\mathrm{D}(\mathrm{A}(\mathfrak{C}_{p}))$, since $\pi(X)\neq 0$ and $\mathrm{D}(\mathrm{A}(\mathfrak{C}_{p}))/T^{D}$ has only trivial localizing ideals. Therefore $1\in \mathrm{Loc}_{\otimes}\langle X, T^{D}\rangle$ and so $$e_{r}\in \mathrm{Loc}_{\otimes}\langle e_{r}\otimes X, e_{r}\otimes T^{D}\rangle\subseteq \mathrm{Loc}_{\otimes}\langle X\rangle\subseteq \mathcal{L},$$ which proves our claim. Then by the choice of $\Delta_{S(\mathcal{L})}$, we have $$\mathrm{Loc}_{\otimes}\langle e_{\Delta_{S(\mathcal{L})}}, \chi_{i,k}\mid i\in S(\mathcal{L}) \rangle\subseteq \mathcal{L}\subseteq \mathrm{Loc}_{\otimes}\langle e_{\Delta_{S(\mathcal{L})}}, \chi_{i,k}\mid i\in S(\mathcal{L}) \rangle$$ and thus $\mathcal{L}=\mathrm{Loc}_{\otimes}\langle e_{\Delta_{S(\mathcal{L})}}, \chi_{i,k}\mid i \in S(\mathcal{L})\rangle$, which proves (2).\\
Note that a subset of $N^{*}$ is open precisely when it avoids $\infty$, or contains $\infty$ together with all but finitely many natural numbers. Moreover, in this space, the Thomason subsets are precisely the open subsets. It follows from (1) and (2) that the homological support induces a bijection between (a) and (d). The bijections between (b),(c) and (d) come from Proposition \ref{prop4.11} and \cite[Theorem 4.10]{Bal05}. The bijection between (d) and (f) comes from \cite[Corollary 2.8]{BKS} and \cite[Theorem 5.7]{Xu26}. The bijection between (e) and (f) follows from Lemma \ref{lem4.12}. This concludes our proof.
\end{proof}

\begin{cor}\label{cor4.16}The telescope conjecture holds for $\mathrm{D}(\mathfrak{C}_{p})$.
\end{cor}
\begin{proof}Let $\mathcal{L}$ be a non-zero localizing ideal of $\mathrm{D}(\mathfrak{C}_{p})$. Then it follows from Theorem \ref{thm4.15} that  $\mathcal{L}$ is torsion or of the form $\mathcal{L}=\mathrm{Loc}_{\otimes}\langle e_{\Delta_{S(\mathcal{L})}}, \chi_{i,k}\mid i \in S(\mathcal{L})\rangle$. If $\mathcal{L}$ is torsion, then by Proposition \ref{Prop2.12}, it is smashing, since $\chi_{i,k}$ is compact for all $i\in \mathbb{N}$.\\
If $\mathcal{L}=\mathrm{Loc}_{\otimes}\langle e_{\Delta_{S(\mathcal{L})}}, \chi_{i,k}\mid i \in S(\mathcal{L})\rangle$, then $\mathcal{L}$ is compactly generated since $e_{n}$ and $\chi_{i,k}$ are compact for all $i,n\in\mathbb{N}$. This concludes the proof.
\end{proof}

\section{\bf Global representations of  $E_{p}$}
In this section, we study the global representations of $E_{p}$, the family of elementary abelian $p$-groups. Although $E_p$ and $\mathfrak{C}_p$ have the same tensor abelian geometry (see \cite[Corollary 5.8]{Xu26}), their tensor triangular geometries are different. A key reason for this difference is that the objects $\chi_{(\mathbb{Z}/p)^{n},k}$ for $n\in \mathbb{N}$ are not  compact in $\mathrm{D}(E_{p})$. This framework is also closely related to the theory of $\mathrm{VI}$-modules. Indeed, Pontryagin duality induces an equivalence between the category of $\mathrm{VI}$-modules (over $k$) and $\mathrm{A}(E_{p})$. We prove that the telescope conjecture holds for $\mathrm{D}(E_{p})$, and hence  for the derived category of $\mathrm{VI}$-modules.

\begin{nota} For simplicity, we write $e_{n,V}$ for $e_{(Z/p)^{n},V}$,  ~$X(n)$ for $X((\mathbb{Z}/p)^{n})$ and so on. We write $i_{n}:\{0,1,\cdots,n-1\}\hookrightarrow \mathbb{N}$ for the inclusion of the down-closed subcategory of $E_{p}$ and $j_{n}:\{n,n+1,\cdots\}\hookrightarrow \mathbb{N}$ for the inclusion of its complement.
We set $$T^{D}:=\mathrm{Loc}_{\otimes}\langle \chi_{i,k}\mid i\in \mathbb{N}\rangle\subseteq \mathrm{D}(E_{p})$$ and $$T^{A}:=\mathrm{Loc}_{\otimes}\langle \chi_{i,k}\mid i\in \mathbb{N}\rangle\subseteq \mathrm{A}(E_{p}).$$
\end{nota}

\begin{lem}\label{lem5.2} Let $X\in \mathrm{A}(E_{p})$. Then we have the following:
\begin{enumerate}
\item $X\in T^{A}$ if and only if $X$ is torsion;

\item If $X$ is not torsion, then $\mathrm{hsupp}(X)$ is cofinite.

\end{enumerate}
\end{lem}
\begin{proof}For part (1), the forward implication is clear. Conversely, suppose $X$ is torsion and write $X$ as a filtered colimit $\operatorname{colim}_i X_i$ of its finitely generated subobjects. By \cite[Lemma 5.5]{Xu26}, it suffices to prove $X_{i}\in P_{\infty}:=\mathrm{Serre}_{\otimes}\langle \chi_{i,k}\mid i\in \mathbb{N}\rangle$ for each $i$. It follows from \cite[Theorem B]{PS22} that there exists some $r_{i}\in \mathbb{N}$  such that  $X_{i}(n)=0$ for $n\geq r_{i}$, since $X_{i}$ is torsion.  This implies that $X_{i}\in P_{\infty}$.\\
Part (2) follows from the definition directly.
\end{proof}

Recall that the Balmer spectrum of $\mathrm{D}(E_{p})^{c}$ is homeomorphic to the Hochster dual of $\mathrm{Spec}(\mathbb{Z})$ and the tensor abelian spectrum of $\mathrm{A}(E_{p})^{c}$ is $\mathbb{N}^{*}$.  Compare the  following result  with  Theorem \ref{thm4.15}.
\begin{thm} \label{thm5.3}Let $\mathcal{L}\subseteq \mathrm{D}(E_{p})$ be a localizing ideal. Then we have the following:
 \begin{enumerate}
\item  If $\mathcal{L}\subseteq T^{D}$, then $\mathcal{L}$ is torsion;

\item If $\mathcal{L}\nsubseteq T^{D}$ , then $S(\mathcal{L})$ is cofinite and $\mathcal{L}=\mathrm{Loc}_{\otimes}\langle e_{\Delta_{S(\mathcal{L})}}, \chi_{i,k}\mid i \in S(\mathcal{L})\rangle$.

\end{enumerate}
 In particular, there are canonical bijections between the following five classes:

\begin{enumerate}
\item[(a)] localizing ideals of $\mathrm{D}(E_{p})$;
\item[(b)] open subsets of $\mathrm{Spc}(\mathrm{A}(E_{p})^{c})$;
\item[(c)] open subsets of $\mathbb{N}^{*}$;
\item[(d)] localizing ideals of $\mathrm{A}(E_{p})$;
\item[(e)] Serre ideals of $\mathrm{A}(E_{p})^{c}$.
\end{enumerate}
\end{thm}
\begin{proof}The proof of part (1) is the same as in Theorem \ref{thm4.15}.  Suppose $\mathcal{L}\nsubseteq T^{D}$. Note that $T^{D}=\{X\in \mathrm{D}(E_{p})\mid H_{i}(X)\in T^{A}, \forall i\in \mathbb{Z}\}$. Then there exist an object $Y\in \mathcal{L}\setminus T^{D}$ and some $d\in \mathbb{Z}$ such that $H_{d}(Y)$ is not torsion. This means that there exist $m\in \mathbb{N}$ and  $y\in H_{d}(Y)(m)$ such that $y$ is a torsion-free element. By \cite[Lemma 7.10]{BBP+25a}, $y$ gives rise to a split monomorphism $e_{m}\rightarrowtail e_{m}\otimes H_{d}(Y)$. On the other hand, as $e_{m}$ is both injective and projective, we have $$\mathrm{Hom}_{\mathrm{D}(E_{p})}(\Sigma^{n}e_{m},e_{m}\otimes Y)\cong \mathrm{Hom}_{\mathrm{A}(E_{p})}(e_{m},H_{n}(e_{m}\otimes Y))$$ and $$\mathrm{Hom}_{\mathrm{D}(E_{p})}(e_{m}\otimes Y,\Sigma^{n}e_{m})\cong \mathrm{Hom}_{\mathrm{A}(E_{p})}(H_{n}(e_{m}\otimes Y),e_{m}).$$ Then we have $e_{m}\in \mathrm{Loc}_{\otimes}\langle e_{m}\otimes Y\rangle\subseteq \mathrm{Loc}_{\otimes}\langle Y\rangle\subseteq \mathcal{L}$, since $e_{m}$ is a retract of the object $H_{n}(e_{m}\otimes Y)\cong e_{m}\otimes H_{n}(Y)$. By \cite[Remark 2.19]{BBP+25b}, the inclusions $i_{m}$ and $j_{m}$ induce a triangle $$(j_{m})_{!}(j_{m})^{*}e_{\Delta_{S(\mathcal{L})}}\rightarrow e_{\Delta_{S(\mathcal{L})}}\rightarrow (i_{m})_{*}(i_{m})^{*}e_{\Delta_{S(\mathcal{L})}}$$ in $\mathrm{D}(E_{p})$. According to \cite[Lemma 4.8]{BBP+25b}, we have $$\mathrm{hsupp}((j_{m})_{!}(j_{m})^{*}e_{\Delta_{S(\mathcal{L})}})=j_{m}(\mathrm{hsupp}((j_{m})^{*}e_{\Delta_{S(\mathcal{L})}}))=j_{m}(j_{m}^{-1}\mathrm{hsupp}(e_{\Delta_{S(\mathcal{L})}}))=\{m,m+1,\cdots\}.$$It then follows from \cite[Lemma 5.5]{BBP+25b} that $(j_{m})_{!}(j_{m})^{*}e_{\Delta_{S(\mathcal{L})}}\in \mathrm{Loc}_{\otimes}\langle e_{m}\rangle\subseteq \mathcal{L}$. On the other hand, $$\mathrm{hsupp}((i_{m})_{*}(i_{m})^{*}e_{\Delta_{S(\mathcal{L})}})=i_{m}(\mathrm{hsupp}((i_{m})^{*}e_{\Delta_{S(\mathcal{L})}}))=i_{m}(i_{m}^{-1}\mathrm{hsupp}(e_{\Delta_{S(\mathcal{L})}}))=\{\Delta_{S(\mathcal{L})},\cdots,m-1\}.$$
Hence $(i_{m})_{*}(i_{m})^{*}e_{\Delta_{S(\mathcal{L})}}\in \mathrm{Loc}_{\otimes}\langle \chi_{i,k}\mid i\in S(\mathcal{L})\rangle\subseteq \mathcal{L}$ and so $e_{\Delta_{S(\mathcal{L})}}\in \mathcal{L}$. This proves the $"\supseteq"$ containment.\\
For the reverse containment, let $Z\in \mathcal{L}$. If $Z$ is torsion, then $Z$ clearly lies in the desired subcategory. Suppose $Z$ is not torsion. Then $\mathrm{hsupp}(Z)$ is cofinite and we write $\Delta_{Z}$ for $\Delta_{\mathrm{hsupp}(Z)}$. By the choice of $\Delta_{S(\mathcal{L})}$, we have $\Delta_{S(\mathcal{L})}\leq \Delta_{Z}$. Consider the triangle $$(j_{\Delta_{Z}})_{!}(j_{\Delta_{Z}})^{*}Z\rightarrow Z\rightarrow (i_{\Delta_{Z}})_{*}(i_{\Delta_{Z}})^{*}Z$$ in $\mathrm{D}(E_{p})$. Then by an argument similar to the one above, we have $(j_{\Delta_{Z}})_{!}(j_{\Delta_{Z}})^{*}Z\in \mathrm{Loc}_{\otimes}\langle e_{\Delta_{Z}}\rangle$ and $(i_{\Delta_{Z}})_{*}(i_{\Delta_{Z}})^{*}Z\in \mathrm{Loc}_{\otimes}\langle e_{\Delta_{S(\mathcal{L})}}, \chi_{i,k}\mid i \in S(\mathcal{L})\rangle$. Moreover, $e_{\Delta_{Z}}\in \mathrm{Loc}_{\otimes}\langle e_{\Delta_{S(\mathcal{L})}}\rangle$, by \cite[Remark 4.12]{PS22}. This implies that $Z\in \mathrm{Loc}_{\otimes}\langle e_{\Delta_{S(\mathcal{L})}}, \chi_{i,k}\mid i \in S(\mathcal{L})\rangle$, which proves the reverse containment.\\
Finally, it follows from (1) and (2) that the homological support induces a bijection between (a) and (c). The bijections between (b), (c) and (d) come from  \cite[Corollary 2.8]{BKS} and \cite[Corollary 5.8]{Xu26}. The bijection between (d) and (e) follows from Lemma \ref{lem4.12}. This concludes our proof.
\end{proof}

The following result is taken from the proof of \cite[Theorem 6.13]{BBP+25b}.
\begin{prop}\label{prop5.4}For all $X,Y\in \mathrm{D}(E_{p})^{c}$, $\mathrm{hsupp}(X)\subseteq \mathrm{hsupp}(Y)$ if and only if $X\in \mathrm{thick}_{\otimes}\langle Y\rangle$.
\end{prop}

\begin{cor}\label{cor5.5} The telescope conjecture holds for $\mathrm{D}(E_{p})$.
\end{cor}
\begin{proof}First we claim that a nonzero torsion localizing ideal is not smashing. Indeed, suppose for contradiction that $\mathcal{L}$ is torsion and smashing. Since $\mathcal{L}$ is smashing, its right orthogonal $\mathcal{L}^\perp$ is a localizing subcategory. On the other hand, the identity
\[
\mathrm{Hom}_{\mathrm{D}(E_{p})}(\chi_{i,k}, \Sigma^{s}e_n)=0
\]
for  $\forall i,n\in\mathbb{N}$, $\forall s\in \mathbb{Z}$ implies that $e_n\in \mathcal{L}^\perp$ for each $n$. Hence $\mathcal{L}^\perp=\mathrm{D}(\mathfrak{C}_{p})$ by \cite[Theorem 7.3]{BBP+25a}, so $\mathcal{L}=0$, contradicting the assumption that $\mathcal{L}$ is non-zero. It remains to prove that the non-torsion localizing ideal $\mathcal{L}=\mathrm{Loc}_{\otimes}\langle e_{\Delta_{S(\mathcal{L})}}, \chi_{i,k}\mid i \in S(\mathcal{L})\rangle$ is compactly generated. It follows from Lemma \ref{lem5.2} that $S(\mathcal{L})$ is cofinite. For each $f\in \mathbb{N}\setminus S(\mathcal{L})$, define $m_{f,k}:=\mathrm{cof}(e_{f,k}\rightarrow 1)$ and set $$m_{\mathcal{L}}:=\displaystyle\bigotimes_{\substack{f\in \mathbb{N}\setminus S(\mathcal{L}) }} m_{f,k}.$$ Then $m_{\mathcal{L}}$ is compact and $\mathrm{hsupp}(m_{\mathcal{L}})=S(\mathcal{L})$. By Proposition \ref{prop5.4}, $e_{\Delta_{S(\mathcal{L})}}\in \mathrm{Loc}_{\otimes}\langle m_{\mathcal{L}}\rangle$. For $i\in S(\mathcal{L})$, we have $\chi_{i,k}\in \mathrm{Loc}_{\otimes}\langle m_{\mathcal{L}}\rangle$, since $\chi_{i,k}$ is a direct summand of $\chi_{i,k}\otimes m_{\mathcal{L}}\otimes V^{*}$ where $V=m_{\mathcal{L}}(i)$.  Hence             $\mathcal{L}\subseteq\mathrm{Loc}_{\otimes}\langle m_{\mathcal{L}}\rangle$. Conversely, we have a triangle $$(j_{\Delta_{S(\mathcal{L})}})_{!}(j_{\Delta_{S(\mathcal{L})}})^{*}m_{\mathcal{L}}\rightarrow m_{\mathcal{L}}\rightarrow (i_{\Delta_{S(\mathcal{L})}})_{*}(i_{\Delta_{S(\mathcal{L})}})^{*}m_{\mathcal{L}}$$ with $(j_{\Delta_{S(\mathcal{L})}})_{!}(j_{\Delta_{S(\mathcal{L})}})^{*}m_{\mathcal{L}}$ and $(i_{\Delta_{S(\mathcal{L})}})_{*}(i_{\Delta_{S(\mathcal{L})}})^{*}m_{\mathcal{L}}$ in $\mathcal{L}$. Hence $m_{\mathcal{L}}\in \mathcal{L}$, and so $\mathcal{L}=\mathrm{Loc}_{\otimes}\langle m_{\mathcal{L}}\rangle$, which is compactly generated. This concludes the proof.
\end{proof}

Note that Pontryagin duality induces an equivalence between the category of $\mathrm{VI}$-modules (over $k$) and $\mathrm{A}(E_{p})$.
\begin{cor}\label{cor5.6}The telescope conjecture holds for the derived category of $\mathrm{VI}$-modules.
\end{cor}

We end this section with the following question, which is supported by the results obtained above.
\begin{quest}Does the telescope conjecture hold for $\mathrm{D}(\mathcal{U})$ for every $U\subseteq \mathcal{G}$ satisfying Hypothesis \ref{hyp2.3}?
\end{quest}
\section{\bf Derived category of  $\mathrm{FI}$-modules}
In this section, we establish the telescope conjecture for the derived category of $\mathrm{FI}$-modules. Although the proof strategy parallels that of the previous section for $\mathrm{VI}$-modules, we provide as many details as possible, for the benefit of readers with an interest in the theory of $\mathrm{FI}$-modules.\\
Let $\lambda\vdash d$ be a partition of an integer $d$ and let $V_{\lambda}$ be the irreducible $k[S_{d}]$-representation labelled by $\lambda$. Define $P_{\lambda}:=M_{d,V_{\lambda}}$. Then \[
P_{\lambda}(n) = \begin{cases}
0 & \text{if } n<d\\
Ind^{S_{n}}_{S_{n}\times S_{n-d}}(V_{\lambda}\boxtimes k) & \text{if } n\geq d.
\end{cases}
\]
Since $k[S_{d}]\cong\displaystyle\bigoplus_{\substack{\lambda\vdash d }} V_{\lambda}^{\oplus \mathrm{dim}(V_{\lambda})}$, $M_{d}\cong \displaystyle\bigoplus_{\substack{\lambda\vdash d }} P_{\lambda}^{\oplus \mathrm{dim}(V_{\lambda})}$. Every finitely generated projective $\mathrm{FI}$-module is a finite direct sum of the $P_\lambda$'s.
\\ The following result follows immediately from the adjunction defining $M_{d,V}$ and Schur's lemma.
\begin{lem}\label{lem6.1} Let $\lambda\vdash d$ and $\mu\vdash e$ be partitions. Then \[
\mathrm{Hom}_{\mathrm{FI}}(P_{\mu},P_{\lambda}) = \begin{cases}
0 & \text{if } e<d\\
0 & \text{if } d=e, \lambda\neq \mu\\
k & \text{if } d=e, \lambda= \mu
\end{cases}
\]
\end{lem}

Fix a partition $\lambda\vdash d$, for any $n\geq d+\lambda_{1}$, there is a padded partition $$\lambda[n]:=(n-d,\lambda_{1},\lambda_{2},\cdots).$$ We write $V_{\lambda[n]}$ for the corresponding irreducible $k[S_{n}]$-representation. The following standard fact about the padded representations $V_{\lambda[n]}$ will be useful.
\begin{lem}\label{lem6.2}Let $\lambda\vdash d$ be a partition and let $n\geq d+\lambda_{1}$. Then for every partition $\mu\vdash e$ with $e\leq d$, we have \[
\mathrm{Hom}_{k[S_{n}]}(V_{\lambda[n]},P_{\mu}(n)) = \begin{cases}
k & \text{if } \mu=\lambda\\
0 & \text{if } \mu\neq\lambda
\end{cases}
\]
\end{lem}
\begin{proof}By definition, $P_{\mu}(n)\cong\mathrm{Ind}^{S_{n}}_{S_{e}\times S_{n-e}}(V_{\mu}\boxtimes k)$ and Pieri's rule says that $V_{\lambda[n]}$ occurs in this induced representation if and only if $\lambda[n]\setminus \mu$ is a horizontal strip of size $n-e$ and then it occurs with multiplicity one. Suppose that $\lambda[n]\setminus \mu$ is a horizontal strip. The usual interlacing criterion for a horizontal strip gives  $\lambda[n]_{i}\geq \mu_{i} \geq \lambda[n]_{i+1}$ for each $i\geq 1$. Since $\lambda[n]_{i+1}=\lambda_{i}$, we have $\mu_{i}\geq \lambda_{i}$ for each $i$. This implies $e\geq d$. It follows that $e=d$ and hence $\mu=\lambda$.

\end{proof}

\begin{lem}\label{lem6.3}Let $X^{\bullet}$ be a perfect complex in $\mathrm{D}_{\mathrm{FI}}$. If there exists some $N\in \mathbb{N}$ such that $X^{\bullet}(n)$ is acyclic for all $n\geq N$, then $X^{\bullet}$ is acyclic.
\end{lem}
\begin{proof}Let $d$ be the maximum of the degrees $|\lambda|$ such that $P_\lambda$ appears as a direct summand of some term of $X^{\bullet}$. By the choice of $d$, we can decompose each term $X^{q}$ of $X$ as $X^{q}=Y^{q}\oplus Z^{q}$ where $Y^{q}:=\displaystyle\bigoplus_{\substack{|\mu|<d }} P_{\mu}^{\oplus m_{q,\mu}}$ and $Z^{q}:=\displaystyle\bigoplus_{\substack{\lambda\vdash d }} P_{\lambda}\otimes_{k}C_{\lambda}^{q}$ for some finite-dimensional vector spaces $C_{\lambda}^{q}$. By Lemma \ref{lem6.1}, we have $d_{X}^{q}(Y^{q})\subseteq Y^{q+1}$ and so $Y^{\bullet}$ is a subcomplex of $X^{\bullet}$. Set $\overline{X}^{\bullet}:=X^{\bullet}/Y^{\bullet}$. It follows from the same lemma that the differentials of $\overline{X}^{\bullet}$ cannot mix distinct partitions $\lambda\vdash d$. Hence $$\overline{X}^{\bullet}\cong \displaystyle\bigoplus_{\substack{\lambda\vdash d }} P_{\lambda}\otimes_{k}C_{\lambda}^{\bullet}$$ where $C_{\lambda}^{\bullet}$ is a bounded complex of finite-dimensional $k$-vector spaces. Choose $n>\mathrm{max}\{N,d+\lambda_{1}\}$. Then by Lemma \ref{lem6.2} we have $$\mathrm{Hom}_{k[S_{n}]}(V_{\lambda[n]}, X^{\bullet}(n))\cong C_{\lambda}^{\bullet}.$$
It follows that $C_{\lambda}^{\bullet}$ is acyclic, since $X^{\bullet}(n)$ is acyclic and $\mathrm{Hom}_{k[S_{n}]}(V_{\lambda[n]}, -)$ is exact. Therefore $\overline{X}^{\bullet}$ is acyclic.\\
Now we have a short exact sequence of complexes $$\xymatrix@C=0.5cm{
  0 \ar[r] & Y^{\bullet} \ar[r] & X^{\bullet} \ar[r] & \overline{X}^{\bullet} \ar[r] & 0 } $$ where $\overline{X}^{\bullet}$ is acyclic, and $X^{\bullet}(n)$ is acyclic for all $n\geq N$. Then $Y^{\bullet}(n)$ is acyclic for all $n\geq N$. Note that $Y^{\bullet}$ has maximal generating degree strictly less than $d$; we can therefore apply induction on $d$. It remains to verify the base case $d=0$. When $d=0$, there is only the empty partition $\emptyset$. Since $P_{\emptyset}=k$, we have $X^{\bullet}\cong k\otimes_{k}C^{\bullet}$, which is acyclic by assumption. This concludes the proof.
\end{proof}
Compare the following result to \cite[Theorem 7.7]{BBP+25a}; this will be crucial in proving the telescope conjecture for the derived category of $\mathrm{FI}$-modules.
\begin{thm}\label{thm6.4}Let $X$ be a non-zero perfect complex in $\mathrm{D}_{\mathrm{FI}}$. Then there exists some $i\in \mathbb{Z}$ such that $H_{i}(X)$ is not torsion.
\end{thm}
\begin{proof}Suppose $H_{i}(X)$ is torsion for all $i\in \mathbb{Z}$. It follows from \cite[Theorem B]{CEFN14} that there exists some $N\in \mathbb{N}$ such that $H_{i}(X(n))\cong H_{i}(X)(n)=0$ for all $i\in \mathbb{Z}$ and $n>N$. Thus $X(n)$ is acyclic for all $n>N$. By Lemma \ref{lem6.3}, it follows that $X$ is acyclic,  contradicting  the assumption that $X$ is non-zero.
\end{proof}

\begin{lem}\label{lem6.5}Let $X\in \mathrm{Mod}_{\mathrm{FI}}$ and let $x\in X(d)$ be torsion-free. Then $x$ gives rise to a split monomorphism $M_{d}\rightarrow M_{d}\otimes X$.
\end{lem}
\begin{proof}For each $[n]$, define $$\alpha_{n}:M_{d}(n)\rightarrow M_{d}(n)\otimes X(n),\quad\quad [f]\mapsto [f]\otimes f_{*}(x).$$ This is natural, hence gives a morphism $\alpha:M_{d}\rightarrow M_{d}\otimes X$. For each $n$, there is a canonical decomposition $$M_{d}(n)\otimes X(n)\cong \bigoplus_{\substack{f\in \mathrm{Inj}([d],[n]) }} X(n)$$ where the summand indexed by $f$ is $k[f]\otimes X(n)$. Under this decomposition, $\alpha_{n}$ is the direct sum of the maps $$k[f]\rightarrow X(n),\quad\quad a[f]\mapsto af_{*}(x).$$ Each of these maps is injective because $f_*(x)\neq 0$ for every $f$ (since $x$ is torsion-free). Hence $\alpha_n$ is injective for every $n$, and so $\alpha$ is injective. By \cite[Theorem 1.5]{GL15}, $M_d$ is injective; therefore $\alpha$ is a split monomorphism.
\end{proof}

We denote by $\mathrm{FI}_{\geq n}$ the full subcategory of $\mathrm{FI}$ consisting of finite sets of cardinality at least
$n$, and by $\mathrm{FI}_{< n}$ the full subcategory consisting of finite sets of cardinality smaller than $n$. Let $i_{n}:\mathrm{FI}_{<n}\hookrightarrow \mathrm{FI}$ be the inclusion of the down-closed subcategory $\mathrm{FI}_{<n}$, with complement $j_{n}:\mathrm{FI}_{\geq n}\hookrightarrow \mathrm{FI}$. We have adjunctions \[
\xymatrix@C=6em@R=0em{
  \mathrm{Fun}(\mathrm{FI}_{<n},\mathrm{Mod}\,k)
  \ar@/^5ex/[r]^{(i_n)_{!}}
  \ar@/_5ex/[r]_{(i_n)_{*}}
  &
  \mathrm{Mod}_{\mathrm{FI}}
  \ar[l]^{(i_n)^*}
}
\]
where $(i_{n})^{*}$ is defined by $(i_{n})^{*}(X)(m)=X(m)$, the left adjoint $(i_{n})_{!}$ of $(i_{n})^{*}$ is given by the left Kan extension along $i_{n}$, and the right adjoint $(i_{n})_{*}$ of $(i_{n})^{*}$ is given by the right Kan extension along $i_{n}$. Passing to the derived categories, we have the following adjunctions \[
\xymatrix@C=6em@R=0em{
  \mathrm{D}(\mathrm{FI}_{<n})
  \ar@/^5ex/[r]^{(i_n)_{!}}
  \ar@/_5ex/[r]_{(i_n)_{*}}
  &
  \mathrm{D}_{\mathrm{FI}}
  \ar[l]^{(i_n)^*}
}
\]
where $\mathrm{D}(\mathrm{FI}_{<n})$ is defined to be the derived category of $\mathrm{Fun}(\mathrm{FI}_{<n},\mathrm{Mod}\,k)$. Similarly, we have adjunctions \[
\xymatrix@C=6em@R=0em{
  \mathrm{D}(\mathrm{FI}_{\geq n})
  \ar@/^5ex/[r]^{(j_n)_{!}}
  \ar@/_5ex/[r]_{(j_n)_{*}}
  &
  \mathrm{D}_{\mathrm{FI}}
  \ar[l]^{(j_n)^*}
}
\]
where $\mathrm{D}(\mathrm{FI}_{\geq n})$ is defined to be the derived category of $\mathrm{Fun}(\mathrm{FI}_{\geq n},\mathrm{Mod}\,k)$. Both $\mathrm{D}(\mathrm{FI}_{<n})$ and $\mathrm{D}(\mathrm{FI}_{\geq n})$ are compactly generated tensor triangulated categories.\\
The proof of each part of the following lemma is standard; we leave the details to the reader.
\begin{lem}\label{lem6.6}For each $n\in \mathbb{N}$, let $i_{n}$ and $j_{n}$ be as above. Then \begin{enumerate}
\item  $(i_{n})^{*}$ and $(j_{n})^{*}$ are symmetric monoidal exact functors;

\item $(i_{n})_{*}$ and $(j_{n})_{!}$ are extensions by zero, and hence they are exact functors and preserve tensor products;

\item $(j_{n})^{*}(j_{n})_{!}\simeq \mathrm{Id}$, $(i_{n})^{*}(j_{n})_{!}\simeq 0$, $(i_{n})^{*}(i_{n})_{*}\simeq \mathrm{Id}$, $(j_{n})^{*}(i_{n})_{*}\simeq 0$;

\item $\mathrm{D}(\mathrm{FI}_{\geq n})=\mathrm{Loc}_{\otimes}\langle M_{n}\rangle$, where $M_{n}$ is seen as an object in $\mathrm{Fun}(\mathrm{FI}_{\geq n},\mathrm{Mod}\,k)$;

\item For any $X\in \mathrm{D}_{\mathrm{FI}}$, there is a triangle $$(j_{n})_{!}(j_{n})^{*}X\rightarrow X\rightarrow (i_{n})_{*}(i_{n})^{*}X.$$

\end{enumerate}
\end{lem}

\begin{df}Let $X\in \mathrm{D}_{\mathrm{FI}}$. The homological support of $X$ is defined by $$\mathrm{hsupp}(X):=\{n\in \mathbb{N}\mid H_{*}(X)(n)\neq 0\}.$$
\end{df}

\begin{rem}\label{rem6.8}One can define homological support for objects in $\mathrm{D}(\mathrm{FI}_{< n})$ and $\mathrm{D}(\mathrm{FI}_{\geq n})$ similarly. For $X,Y\in \mathrm{D}(\mathrm{FI}_{< n})$, $\mathrm{hsupp}(X)\subseteq \mathrm{hsupp}(Y)$ if and only if $X\in \mathrm{Loc}_{\otimes}\langle Y\rangle$.
\end{rem}
\begin{lem}\label{lem6.9}Let $X\in\mathrm{D}_{\mathrm{FI}}$ be non-zero. Then we have $X\in \mathrm{Loc}_{\otimes}\langle M_{\min\{\mathrm{hsupp}(X)\}}\rangle$. If $X$ is compact, then the same inclusion holds with the thick ideal in place of the localizing ideal.
\end{lem}
\begin{proof}Set $n:=\min\{\mathrm{hsupp}(X)\}$. By Lemma \ref{lem6.6} there exists a triangle $$(j_{n})_{!}(j_{n})^{*}X\rightarrow X\rightarrow (i_{n})_{*}(i_{n})^{*}X.$$ By the definition of $\mathrm{hsupp}$ and the minimality of $n$, $(i_{n})^{*}X\cong 0$ and so $X\cong (j_{n})_{!}(j_{n})^{*}X$. It follows from the same lemma that $$(j_{n})_{!}(j_{n})^{*}X\in (j_{n})_{!}\mathrm{Loc}_{\otimes}\langle M_{n}\rangle\subseteq \mathrm{Loc}_{\otimes}\langle (j_{n})_{!}M_{n}\rangle=\mathrm{Loc}_{\otimes}\langle M_{n}\rangle.$$ This proves the first claim. The claim about compact $X$ is a formal consequence of the first claim and the Neeman-Thomason localization theorem.
\end{proof}

\begin{prop}\label{prop6.10}For $X\in \mathrm{D}_{\mathrm{FI}}$, if there exist some $n\in \mathbb{Z}$, $m\in \mathbb{N}$ and $x\in H_{n}(X)(m)$ such that $x$ is not torsion, then $M_{m}\in \mathrm{thick}_{\otimes}\langle X\rangle$.
\end{prop}
\begin{proof}By assumption and Lemma \ref{lem6.5}, $M_{m}$ is a retract of $M_{m}\otimes H_{n}(X)$. On the other hand, as $M_{m}$ is both injective and projective, we have $$\mathrm{Hom}_{\mathrm{D}_{\mathrm{FI}}}(\Sigma^{n}M_{m},M_{m}\otimes X)\cong \mathrm{Hom}_{\mathrm{Mod}_{\mathrm{FI}}}(M_{m},H_{n}(M_{m}\otimes X))$$ and $$\mathrm{Hom}_{\mathrm{D}_{\mathrm{FI}}}(M_{m}\otimes X,\Sigma^{n}M_{m})\cong \mathrm{Hom}_{\mathrm{Mod}_{\mathrm{FI}}}(H_{n}(M_{m}\otimes X),M_{m}).$$ Then we have $M_{m}\in \mathrm{thick}_{\otimes}\langle M_{m}\otimes X\rangle\subseteq \mathrm{thick}_{\otimes}\langle X\rangle$, since $M_{m}$ is a retract of the object $H_{n}(M_{m}\otimes X)\cong M_{m}\otimes H_{n}(X)$.
\end{proof}

\begin{prop}\label{prop6.11}Let $X,Y\in \mathrm{D}_{\mathrm{FI}}^{c}$. Then $\mathrm{hsupp}(X)\subseteq \mathrm{hsupp}(Y)$ if and only if $X\in \mathrm{thick}_{\otimes}\langle Y\rangle$.
\end{prop}
\begin{proof}The $"\Leftarrow"$ implication is clear. For the reverse implication, suppose $\mathrm{hsupp}(X)\subseteq \mathrm{hsupp}(Y)$. By Theorem \ref{thm6.4} and Proposition \ref{prop6.10}, there exist $n_{1},n_{2}\in \mathbb{N}$ such that $M_{n_{1}}\in \mathrm{thick}_{\otimes}\langle X\rangle$ and $M_{n_{2}}\in \mathrm{thick}_{\otimes}\langle Y\rangle$. Set $n:=\max\{n_{1},n_{2}\}$.  Consider the restricted functor $$(i_{n})^{*}:\mathrm{D}_{\mathrm{FI}}^{c}\rightarrow \mathrm{D}(\mathrm{FI}_{<n})^{c}.$$ It follows from Lemma \ref{lem6.9} that $\mathrm{Ker}((i_{n})^{*})=\mathrm{thick}_{\otimes}\langle M_{n}\rangle\subseteq \mathrm{D}_{\mathrm{FI}}^{c}$, and there is a tensor equivalence $$\mathrm{D}_{\mathrm{FI}}^{c}/\mathrm{thick}_{\otimes}\langle M_{n}\rangle\simeq \mathrm{D}(\mathrm{FI}_{<n})^{c}.$$ We also have $(i_{n})^{*}X\in \mathrm{thick}_{\otimes}\langle (i_{n})^{*}Y \rangle$, since $\mathrm{hsupp}((i_{n})^{*}X)\subseteq \mathrm{hsupp}((i_{n})^{*}Y)$. This implies that $X\in \mathrm{thick}_{\otimes}\langle Y, M_{n}\rangle\subseteq \mathrm{D}_{\mathrm{FI}}^{c}$. Since $M_{n}\in \mathrm{thick}_{\otimes}\langle M_{n_{2}}\rangle\subseteq\mathrm{thick}_{\otimes}\langle Y\rangle$, we have $X\in \mathrm{thick}_{\otimes}\langle Y\rangle$, which concludes the proof.

\end{proof}

We denote by $M_{n,-}: \mathrm{D}(k[S_{n}])\rightarrow \mathrm{D}_{\mathrm{FI}}$ the derived functor of $M_{n,-}: \mathrm{Mod}_{k[S_{n}]}\rightarrow \mathrm{Mod}_{\mathrm{FI}}$. It is the left adjoint of the evaluation functor $ev_{n}:\mathrm{D}_{\mathrm{FI}}\rightarrow \mathrm{D}(k[S_{n}]),\quad X\mapsto X(n)$. The object $\chi_{n,k}$ is defined as follows: $\chi_{n,k}(n)=k$ and $\chi_{n,k}(m)=0$ for $m\neq n$. For $V\in \mathrm{D}(k[S_{n}])$, set $\chi_{n,V}:=M_{n,V}\otimes \chi_{n,k}$. An argument similar to the one in the proof of Lemma \ref{lem3.1} shows that $\mathrm{Loc}_{\otimes}\langle \chi_{n,V}\rangle=\mathrm{Loc}_{\otimes}\langle \chi_{n,k}\rangle\subseteq \mathrm{D}_{\mathrm{FI}}$ whenever $V$ is non-zero. For a localizing ideal $\mathcal{L}\subseteq \mathrm{D}_{\mathrm{FI}}$, we write $S(\mathcal{L}):=\{n\in \mathbb{N}\mid \chi_{n,k}\in \mathcal{L}\}$.   We set $$T^{D}:=\mathrm{Loc}_{\otimes}\langle \chi_{i,k}\mid i\in \mathbb{N}\rangle\subseteq \mathrm{D}_{\mathrm{FI}}$$ and $$T^{A}:=\mathrm{Loc}_{\otimes}\langle \chi_{i,k}\mid i\in \mathbb{N}\rangle\subseteq \mathrm{Mod}_{\mathrm{FI}}.$$

We omit the proof of the following lemma, which is the FI-analogue of Lemma \ref{lem5.2}.
\begin{lem}\label{lem6.12} Let $X\in \mathrm{Mod}_{\mathrm{FI}}$. Then we have the following:
\begin{enumerate}
\item $X\in T^{A}$ if and only if $X$ is torsion;

\item If $X$ is not torsion, then $\mathrm{hsupp}(X)$ is cofinite.

\end{enumerate}
\end{lem}

\begin{thm}\label{thm6.13} Let $\mathcal{L}\subseteq \mathrm{D}_{\mathrm{FI}}$ be a localizing ideal. Then we have the following:
 \begin{enumerate}
\item  If $\mathcal{L}\subseteq T^{D}$, then $\mathcal{L}$ is torsion;

\item If $\mathcal{L}\nsubseteq T^{D}$ , then $S(\mathcal{L})$ is cofinite and $\mathcal{L}=\mathrm{Loc}_{\otimes}\langle M_{\Delta_{S(\mathcal{L})}}, \chi_{i,k}\mid i \in S(\mathcal{L})\rangle$.

\end{enumerate}
 In particular, there are canonical bijections between the following five classes:

\begin{enumerate}
\item[(a)] localizing ideals of $\mathrm{D}_{\mathrm{FI}}$;
\item[(b)] open subsets of $\mathrm{Spc}(\mathrm{Mod}_{\mathrm{FI}})^{c})$;
\item[(c)] open subsets of $\mathbb{N}^{*}$;
\item[(d)] localizing ideals of $\mathrm{Mod}_{\mathrm{FI}}$;
\item[(e)] Serre ideals of $\mathrm{Mod}_{\mathrm{FI}}^{c}$.
\end{enumerate}
\end{thm}
\begin{proof}The proof of part (1) is the same as in Theorem \ref{thm4.15}. Suppose $\mathcal{L}\nsubseteq T^{D}$. Then there exists an object $Y\in \mathcal{L}\setminus T^{D}$.  By the tensor equivalence $\mathrm{D}_{\mathrm{FI}}/T^{D}\simeq \mathrm{D}(\mathrm{Mod}_{\mathrm{FI}}/T^{A})$, there exists some $d\in \mathbb{Z}$ such that $H_{d}(Y)$ is not torsion. This means that there exist $m\in \mathbb{N}$ and  $y\in H_{d}(Y)(m)$ such that $y$ is a torsion-free element. It follows from Proposition \ref{prop6.10} that we have $M_{m}\in \mathrm{Loc}_{\otimes}\langle Y\rangle\subseteq \mathcal{L}$ and thus $S(\mathcal{L})$ is cofinite. By Lemma \ref{lem6.6}, the inclusions $i_{m}$ and $j_{m}$ induce a triangle $$(j_{m})_{!}(j_{m})^{*}M_{\Delta_{S(\mathcal{L})}}\rightarrow M_{\Delta_{S(\mathcal{L})}}\rightarrow (i_{m})_{*}(i_{m})^{*}M_{\Delta_{S(\mathcal{L})}}$$ in $\mathrm{D}_{\mathrm{FI}}$ with $(j_{m})_{!}(j_{m})^{*}M_{\Delta_{S(\mathcal{L})}}\in \mathrm{Loc}_{\otimes}\langle M_{m}\rangle\subseteq \mathcal{L}$ and  $(i_{m})_{*}(i_{m})^{*}M_{\Delta_{S(\mathcal{L})}}\in \mathrm{Loc}_{\otimes}\langle \chi_{i,k}\mid i\in S(\mathcal{L})\rangle\subseteq \mathcal{L}$. Thus  $M_{\Delta_{S(\mathcal{L})}}\in \mathcal{L}$. This proves the $"\supseteq"$ containment.\\
For the reverse containment, let $Z\in \mathcal{L}$. If $Z$ is torsion, then $Z$ clearly lies in the desired subcategory. Suppose $Z$ is not torsion. Then $\mathrm{hsupp}(Z)$ is cofinite and we write $\Delta_{Z}$ for $\Delta_{\mathrm{hsupp}(Z)}$. By the choice of $\Delta_{S(\mathcal{L})}$, we have $\Delta_{S(\mathcal{L})}\leq \Delta_{Z}$. Consider the triangle $$(j_{\Delta_{Z}})_{!}(j_{\Delta_{Z}})^{*}Z\rightarrow Z\rightarrow (i_{\Delta_{Z}})_{*}(i_{\Delta_{Z}})^{*}Z$$ in $\mathrm{D}_{\mathrm{FI}}$. Then by an argument similar to the one above, we have $(j_{\Delta_{Z}})_{!}(j_{\Delta_{Z}})^{*}Z\in \mathrm{Loc}_{\otimes}\langle M_{\Delta_{Z}}\rangle$ and $(i_{\Delta_{Z}})_{*}(i_{\Delta_{Z}})^{*}Z\in \mathrm{Loc}_{\otimes}\langle M_{\Delta_{S(\mathcal{L})}}, \chi_{i,k}\mid i \in S(\mathcal{L})\rangle$. Since $M_{\Delta_{Z}}\in \mathrm{Loc}_{\otimes}\langle M_{\Delta_{S(\mathcal{L})}}\rangle$, we have $Z\in \mathrm{Loc}_{\otimes}\langle M_{\Delta_{S(\mathcal{L})}}, \chi_{i,k}\mid i \in S(\mathcal{L})\rangle$, which proves the reverse containment.\\
Finally, it follows from (1) and (2) that the homological support induces a bijection between (a) and (c). The bijections between (b), (c) and (d) come from  \cite[Corollary 2.8]{BKS} and \cite[Remark 5.10]{Xu26}. The bijection between (d) and (e) follows from Lemma \ref{lem4.12}. This concludes our proof.
\end{proof}

\begin{cor} The telescope conjecture holds for $\mathrm{D}_{\mathrm{FI}}$.
\end{cor}
\begin{proof}By the same argument as in the proof of Corollary \ref{cor5.5}, one can show every non-zero torsion localizing ideal of $\mathrm{D}_{\mathrm{FI}}$ is not smashing. It remains to prove that the non-torsion localizing ideal $\mathcal{L}=\mathrm{Loc}_{\otimes}\langle M_{\Delta_{S(\mathcal{L})}}, \chi_{i,k}\mid i \in S(\mathcal{L})\rangle$ is compactly generated. Note that $S(\mathcal{L})$ is cofinite. For each $f\in \mathbb{N}\setminus S(\mathcal{L})$, define $m_{f,k}:=\mathrm{cof}(M_{f,k}\rightarrow 1)$ and set $$m_{\mathcal{L}}:=\displaystyle\bigotimes_{\substack{f\in \mathbb{N}\setminus S(\mathcal{L}) }} m_{f,k}.$$ Then $m_{\mathcal{L}}$ is compact and $\mathrm{hsupp}(m_{\mathcal{L}})=S(\mathcal{L})$. By Proposition \ref{prop6.11}, $M_{\Delta_{S(\mathcal{L})}}\in \mathrm{Loc}_{\otimes}\langle m_{\mathcal{L}}\rangle$.  For $i\in S(\mathcal{L})$, we have $\chi_{i,k}\in \mathrm{Loc}_{\otimes}\langle m_{\mathcal{L}}\rangle$, since $\chi_{i,k}$ is a direct summand of $\chi_{i,k}\otimes m_{\mathcal{L}}\otimes V^{*}$ where $V=m_{\mathcal{L}}(i)$. Hence             $\mathcal{L}\subseteq\mathrm{Loc}_{\otimes}\langle m_{\mathcal{L}}\rangle$. Conversely, we have a triangle $$(j_{\Delta_{S(\mathcal{L})}})_{!}(j_{\Delta_{S(\mathcal{L})}})^{*}m_{\mathcal{L}}\rightarrow m_{\mathcal{L}}\rightarrow (i_{\Delta_{S(\mathcal{L})}})_{*}(i_{\Delta_{S(\mathcal{L})}})^{*}m_{\mathcal{L}}$$ with $(j_{\Delta_{S(\mathcal{L})}})_{!}(j_{\Delta_{S(\mathcal{L})}})^{*}m_{\mathcal{L}}$ and $(i_{\Delta_{S(\mathcal{L})}})_{*}(i_{\Delta_{S(\mathcal{L})}})^{*}m_{\mathcal{L}}$ in $\mathcal{L}$. Hence $m_{\mathcal{L}}\in \mathcal{L}$, and so $\mathcal{L}=\mathrm{Loc}_{\otimes}\langle m_{\mathcal{L}}\rangle$, which is compactly generated. This concludes the proof.
\end{proof}

\renewcommand\refname
\vspace{4mm}

\textbf{Peng Xu}\\
School of Mathematics, Nanjing University,
Nanjing 210093, P. R. China;\\
E-mail: \textsf{602023210017@smail.nju.edu.cn}\\[1mm]

\end{document}